\documentclass[12pt]{amsart}
\usepackage{amsmath,amsthm,amssymb,amscd,amsfonts}
\usepackage{graphicx}
\usepackage{microtype}
\usepackage{tikz}
\usepackage{latexsym}
\ifdefined\MonochromeFigure
  \colorlet{contourRight}{black}
  \colorlet{contourTop}{black}
  \colorlet{contourLeft}{black}
  \colorlet{contourBottom}{black}
\else
  \colorlet{contourRight}{blue}
  \colorlet{contourTop}{red}
  \colorlet{contourLeft}{green!55!black}
  \colorlet{contourBottom}{orange!90!black}
\fi

\ifdefined\IJNTReviewCopy
  \usepackage{setspace}
\fi

\usepackage[colorlinks=true,linkcolor=blue,citecolor=blue,urlcolor=blue]{hyperref}

\numberwithin{equation}{section}

\newtheorem{thm}{Theorem}[section]

\newtheorem{lem}[thm]{Lemma}
\newtheorem{prop}[thm]{Proposition}
\theoremstyle{definition}

\theoremstyle{remark}
\newtheorem{rem}[thm]{Remark}

\newcommand{\Real}{\mathbb R}
\newcommand{\Complex}{\mathbb C}

\newcommand{\abs}[1]{\left\vert#1\right\vert}

\newcommand{\set}[1]{\left\{#1\right\}}
\newcommand{\intpminf}{\int_{-\infty}^{\infty}}
\newcommand{\intzi}{\int_{0}^{\infty}}
\newcommand{\sumonei}{\sum_{n=1}^\infty}
\newcommand{\cmpint}{\int_{\sigma-i\infty}^{\sigma+i\infty}}

\newcommand{\zetas}{\zeta(s)}
\newcommand{\Gammas}{\Gamma(s)}
\newcommand{\dds}{\frac{d}{ds}}
\newcommand{\half}{\frac{1}{2}}
\newcommand{\Res}{\operatorname*{Res}}
\newcommand{\ber}{\operatorname{ber}}
\newcommand{\bei}{\operatorname{bei}}

\hypersetup{pdftitle={The Mobius Laplace transform: explicit formulas, RH, and simple zeros},pdfauthor={Sergey Liflandsky},pdfsubject={Analytic number theory},pdfkeywords={Riemann zeta function, Mobius function, explicit formula, Riemann hypothesis, Bessel functions}}
\begin{document}

\righthyphenmin=8
\emergencystretch=2em

\title[The M\"obius Laplace transform and simple zeros]{Explicit formulas for the M\"obius Laplace transform and a criterion for the Riemann hypothesis and simple zeros}
\author[Sergey Liflandsky]{Sergey Liflandsky}
\thanks{Corresponding author: Sergey Liflandsky. Email: \href{mailto:sliflandsky@gmail.com}{sliflandsky@gmail.com}.}
\email{sliflandsky@gmail.com}

\begin{abstract}
The estimate $\Phi(e^{-x})=O(x^{-1/2})$ as $x\to0^{+}$ for the discrete Laplace transform of the M\"obius function implies both the Riemann hypothesis and simplicity of every nontrivial zero. We give a direct proof and an explicit formula for the transform. Collisions between Gamma poles and trivial zeta zeros produce double poles, logarithmic residues, and four entire functions. Their M\"obius-weighted Bessel representations combine into a single regularized modified-Bessel series with a positive integral representation and an explicit truncation bound. Under the Riemann hypothesis and simplicity, the endpoint estimate is equivalent to uniform boundedness of finite Fej\'er-weighted zero sums; absolute summability is a sufficient alternative hypothesis. We also obtain multiplicity bounds from logarithmic losses, necessary square-summability of the residues, and a quantitative contour remainder. The relation to earlier convolution and summation theorems is stated explicitly.

\end{abstract}

\maketitle
\ifdefined\IJNTReviewCopy\thispagestyle{headings}\fi

\section{The Mellin integral representation}\label{sec:mellin}

Let
\begin{equation}\label{eq:phidef}
\Phi(x)=\sum_{n=1}^{\infty}\mu(n)x^{n}\qquad(\abs{x}<1),
\end{equation}
where $\mu$ is the M\"obius function. We study the discrete Laplace transform $\Phi(e^{-x})$ for $x>0$. Its endpoint estimate $O(x^{-1/2})$ forces both the Riemann hypothesis and simple nontrivial zeros. The explicit formulas below identify the contribution of the trivial poles in terms of $\kappa,\beta,\lambda,\upsilon$, and then as one M\"obius-weighted Bessel series.

For $\operatorname{Re}s>1$, the Dirichlet series \cite{Davenport} is
\begin{equation}\label{eq:dirichlet}
\frac{1}{\zeta(s)}=\sum_{n=1}^{\infty}\frac{\mu(n)}{n^{s}}.
\end{equation}
Since
\begin{equation}\label{eq:gammaint}
\intzi e^{-nt}t^{s-1}\,dt=\frac{\Gamma(s)}{n^s}
\qquad(\operatorname{Re}s>0),
\end{equation}
and $\sum_n|\mu(n)|\Gamma(\sigma)n^{-\sigma}\leq\Gamma(\sigma)\zeta(\sigma)<\infty$ for $\sigma>1$, absolute convergence permits termwise integration \cite{Rudin}:
\begin{equation}\label{eq:mellinfwd}
\frac{\Gamma(s)}{\zeta(s)}
=\intzi\Big(\sumonei\mu(n)e^{-nt}\Big)t^{s-1}\,dt
=\intzi\Phi(e^{-t})t^{s-1}\,dt.
\end{equation}
Mellin inversion therefore gives
\begin{equation}\label{eq:mellininv}
\Phi(e^{-x})=\frac{1}{2\pi i}\cmpint
 \frac{\Gamma(s)}{\zeta(s)}x^{-s}\,ds
\qquad(\sigma>1,\ x>0).
\end{equation}
We verify the needed convergence before shifting this line.

For fixed $\sigma$, Stirling's formula \cite[Section~5.11]{DLMF} gives
\begin{equation}\label{eq:stirling}
\abs{\Gamma(\sigma+it)}\sim\sqrt{2\pi}\,
 e^{-\pi\abs{t}/2}\abs{t}^{\sigma-1/2}
\qquad(\abs t\to\infty).
\end{equation}
Euler's product and $|p^{\sigma+it}-1|\leq p^\sigma+1$ imply
\begin{equation}\label{eq:zetalow}
\abs{\zeta(\sigma+it)}\geq\prod_p\frac{p^\sigma}{p^\sigma+1}
=\frac{\zeta(2\sigma)}{\zeta(\sigma)}=: \delta(\sigma)>0
\qquad(\sigma>1).
\end{equation}
Consequently,
\begin{equation}\label{eq:absint}
\frac1{2\pi}\int_{-\infty}^{\infty}
 \abs{\frac{\Gamma(\sigma+it)}{\zeta(\sigma+it)}x^{-\sigma-it}}\,dt
\leq\frac{x^{-\sigma}}{2\pi\delta(\sigma)}
 \int_{-\infty}^{\infty}\abs{\Gamma(\sigma+it)}\,dt<\infty.
\end{equation}
Here $\Gamma$ is bounded on $|t|\leq1$; on the tails, \eqref{eq:stirling} yields
\begin{equation}\label{eq:gammaeval}
\int_{|t|\geq1}|\Gamma(\sigma+it)|\,dt
\leq\frac{2C(\sigma)\Gamma(\sigma+\tfrac12)}{(\pi/2)^{\sigma+1/2}}.
\end{equation}
The function $t\mapsto\Phi(e^{-t})$ and its derivative converge locally uniformly on $(0,\infty)$, by comparison with $\sum_n e^{-nt_0}$ and $\sum_n ne^{-nt_0}$ for $t\geq t_0>0$. Its weighted absolute integral is at most $\Gamma(\sigma)\zeta(\sigma)$. Thus Lemma~\ref{lem:mellininv} in Appendix~\ref{app:mellininv} applies and proves \eqref{eq:mellininv}. The elementary bound $|\Phi(e^{-x})|\leq(e^x-1)^{-1}\leq x^{-1}$ will also be used.

\subsection{Statement of the main results}\label{subsec:results}

Write $\rho=\beta_\rho+i\gamma$ for a nontrivial zero of $\zeta$, and let \textbf{(S)} denote the assertion that every such zero is simple. The computation of Platt and Trudgian \cite{PlattTrudgian} verifies RH and simplicity up to height $3\cdot10^{12}$; no finite verification is assumed in our proofs.

Unless absolute convergence is asserted, sums over zeros use selected symmetric cutoffs:
\begin{equation}\label{eq:symconv}
\sum_\rho a_\rho:=\lim_{\nu\to\infty}
 \sum_{\rho:\,|\gamma|<T_\nu}a_\rho.
\end{equation}
Here $T_\nu\in[\nu,\nu+1]$ are the heights of Lemma~\ref{lem:ordinates}, chosen so that the horizontal lines avoid all zeros quantitatively. This is the convention used in classical explicit formulas \cite[Theorem~14.27]{Titchmarsh}. Under Hypothesis (H) in Section~\ref{sec:zerosum}, the zero series is absolutely convergent.

We shall prove the following results. Define the entire functions
\begin{align}
\kappa(x)&=\sum_{m=0}^{\infty}\frac{(-1)^{m}(2\pi x)^{2m+1}}{\zeta(2m+2)\,[(2m+1)!]^{2}},\notag\\
\beta(x)&=\sum_{m=1}^{\infty}\frac{(-1)^{m}(2\pi x)^{2m}}{[(2m)!]^{2}\,\zeta(2m+1)},\label{eq:defkappabeta}\\
\lambda(x)&=\sum_{m=1}^{\infty}\frac{(-1)^{m}\,\zeta'(2m+1)\,(2\pi x)^{2m}}{[(2m)!]^{2}\,\zeta(2m+1)^{2}},\notag\\
\upsilon(x)&=\sum_{m=1}^{\infty}\frac{(-1)^{m}\,\Gamma'(2m+1)\,(2\pi x)^{2m}}{[(2m)!]^{3}\,\zeta(2m+1)}.\label{eq:deflamups}
\end{align}

\begin{thm}[Explicit formula]\label{thm:main}
Assume Hypothesis (S). Then for every $x>0$, with the convention \eqref{eq:symconv},
\begin{multline}\label{eq:mainformula}
\sum_{n=1}^{\infty}\mu(n)e^{-nx}
=\sum_{\rho}\frac{\Gamma(\rho)\,x^{-\rho}}{\zeta'(\rho)}\\
+\pi\kappa(x)+2\lambda(x)+4\upsilon(x)-2\beta(x)\log(2\pi x)-2 .
\end{multline}
Equivalently, the complete negative-integer contribution is a single absolutely convergent series:
\begin{multline}\label{eq:explicit-unified}
\Phi(e^{-x})=\sum_\rho\frac{\Gamma(\rho)x^{-\rho}}{\zeta'(\rho)}-2\\
+\sum_{k=1}^{\infty}\frac{\mu(k)}k
\left[4\operatorname{Re}K_0\left(2e^{i\pi/4}\sqrt{\frac{2\pi x}{k}}\right)
+2\log\left(\frac{2\pi x}{k}\right)+4\gamma_e\right].
\end{multline}
Here $K_0$ is the principal branch of the modified Bessel function of order zero and $\gamma_e$ is Euler's constant. The square brackets are summed as a whole. Without (S), replace each simple-zero term by the residue at that distinct zero, as in Remark~\ref{rem:mult}.
\end{thm}

\begin{thm}[M\"obius--Bessel expansions]\label{thm:bessel}
For every $x>0$ the following series converge absolutely:
\begin{equation}\label{eq:besselcombined}
\sum_{k=1}^{\infty}\frac{\mu(k)}{k}\left(J_{0}\Big(2e^{-i\pi/4}\sqrt{\tfrac{2\pi x}{k}}\Big)-1\right)=\beta(x)+i\,\kappa(x),
\end{equation}
\begin{equation}\label{eq:bessellambda}
\sum_{k=1}^{\infty}\frac{\mu(k)\log k}{k}\,\operatorname{Re}\!\left(J_{0}\Big(2e^{-i\pi/4}\sqrt{\tfrac{2\pi x}{k}}\Big)-1\right)=\lambda(x),
\end{equation}
where $J_{0}$ is the Bessel function of the first kind of order zero. Moreover $\kappa,\beta,\lambda,\upsilon$ are entire functions. The representation of the remaining function $\upsilon$ and the single combined formula for all four contributions are given in Proposition~\ref{prop:upsilon} and Theorem~\ref{thm:unified}.
\end{thm}

\begin{thm}[A criterion for RH and simplicity]\label{thm:criterion}
\leavevmode\par
\begin{enumerate}
\item[(a)] The estimate
\begin{equation}\label{eq:rhbound}
\Phi(e^{-x})=\sum_{n=1}^{\infty}\mu(n)e^{-nx}
 =O(x^{-1/2})\qquad(x\to0^{+})
\end{equation}
implies the Riemann hypothesis and simplicity of every nontrivial zero. No assumption on the zeros is required.
\item[(b)] Under RH and (S), \eqref{eq:rhbound} holds if and only if the finite sums
\[
\sum_{|\gamma|<T}\left(1-\frac{|\gamma|}{T}\right)
 \frac{\Gamma(\rho)}{\zeta'(\rho)}x^{-i\gamma}
\]
are uniformly bounded for $T>0$ and $x>0$. This is Hypothesis $(H_{\mathrm F})$ of Section~\ref{sec:zerosum}; the absolute summability condition (H) is sufficient.
\end{enumerate}
\end{thm}

Thus \eqref{eq:rhbound} is equivalent to RH, (S), and $(H_{\mathrm F})$ together. The additional uniform bound is not derived from RH and simplicity alone. Theorem~\ref{thm:fejer} identifies its best constant, and Proposition~\ref{prop:multlog} shows that even $\Phi(e^{-x})=o(x^{-1/2}\log(1/x))$ implies RH and simplicity.

Mellin inversion and contour shifting belong to the classical approach of Hardy and Littlewood \cite{HL1916}, whose Gaussian-kernel series also appears in Ramanujan's notebooks \cite[pp.~468--469]{BerndtIV}. Related Riesz-type criteria are studied in \cite{Riesz1916,AGM2022,GM2023}. B\'aez-Duarte's convolution theorem \cite{BaezDuarte} contains the endpoint implication to RH and simplicity; its exact specialization is given in Section~\ref{sec:criterion}. General reciprocal-function identities were developed by K\"uhn, Robles and Roy \cite{KRR2015}. The real simple-zero identity here is also a specialization of the M\"obius summation theorem of Chorge and Dixit \cite{ChorgeDixit}; see Remark~\ref{rem:chorge-dixit}.

Our focus is the explicit evaluation and quantitative analysis of this transform. Sections~\ref{sec:critstrip}--\ref{sec:even} compute every residue, including the logarithmic terms caused by double poles; Appendices~\ref{app:gammalaurent}--\ref{app:sanity} check the even-pole calculation independently. Section~\ref{sec:integral} supplies the contour estimates and an explicit remainder. Section~\ref{sec:special} completes the representation of $\upsilon$, derives the unified formula coefficient by coefficient, and bounds its arithmetic tail. Sections~\ref{sec:zerosum}--\ref{sec:criterion} establish the Fej\'er characterization, multiplicity bounds, and necessary residue estimates. These refinements are distinguished from the general summation and convolution principles cited above.

For the Mertens function, see Titchmarsh \cite[Theorem~14.27]{Titchmarsh}, Bartz \cite{Bartz1991}, and Inoue \cite{Inoue2018}.

For comparison we record the Mertens-function formula: for $x>1$, under RH and Hypothesis (S),
\begin{equation}\label{eq:titchmarsh}
\sideset{}{'}\sum_{n\le x}\mu(n)=\lim_{\nu\to\infty}\sum_{\abs{\gamma}<T_{\nu}}\frac{x^{\rho}}{\rho\,\zeta'(\rho)}
-2+\sum_{n=1}^{\infty}\frac{(-1)^{n-1}(2\pi/x)^{2n}}{(2n)!\,n\,\zeta(2n+1)}\, .
\end{equation}
The heights are selected as in the cited theorem, and the prime assigns weight $1/2$ to $n=x$ at integer $x$. Unlike \eqref{eq:mainformula}, its trivial-zero terms come from simple poles. The residue at $s=0$ contributes $-2$ in both formulas.

\section{Calculation of the residues in the critical strip}\label{sec:critstrip}

At a simple nontrivial zero $\rho$, the local Taylor expansion of $\zeta$ gives
\begin{equation}\label{eq:resrho}
\Res_{s=\rho}\ \frac{\Gamma(s)x^{-s}}{\zetas}=\frac{\Gamma(\rho)\,x^{-\rho}}{\zeta'(\rho)} .
\end{equation}

Next, $\Gammas$ has a simple pole at $s=0$, and $\zeta(0)=-\frac{1}{2}$. Let us first calculate the residues of $\Gammas$ at the nonpositive integers. We use the reflection identity (\cite[Section 5.5]{DLMF})
\begin{equation}\label{eq:reflection}
\Gammas\,\Gamma(1-s)=\frac{\pi}{\sin(\pi s)} .
\end{equation}
Since $\Gammas$ has no zeros, and since $\sin(\pi s)=(-1)^{n}\sin\big(\pi(s+n)\big)$ near $s=-n$,
\begin{equation}\label{eq:sinlimit}
\lim_{s\to-n}\frac{(s+n)\,\pi}{\sin(\pi s)}=\frac{1}{\cos(\pi n)}=(-1)^{n} ,
\end{equation}
and we obtain that
\begin{equation}\label{eq:resgamma}
\Res_{s=-n}\ \Gammas=\lim_{s\to-n}\frac{(s+n)\,\pi}{\sin(\pi s)\,\Gamma(1-s)}=\frac{(-1)^{n}}{\Gamma(1+n)}=\frac{(-1)^{n}}{n!} .
\end{equation}
More precisely, we shall use the Laurent expansion at $s=-n$: 
\begin{equation}\label{eq:gammalaurent}
\Gamma(s)=\frac{(-1)^{n}}{n!}\left(\frac{1}{s+n}+\psi(n+1)+O(s+n)\right)\qquad(s\to-n),
\end{equation}
where $\psi=\Gamma'/\Gamma$ is the digamma function; the constant term follows from \eqref{eq:reflection} and the expansion of $\pi/\sin(\pi s)$ or, alternatively, from $\Gamma(s)=\Gamma(s+n+1)/\big(s(s+1)\cdots(s+n)\big)$; see \cite[Section 1.1]{Lebedev}.

By \eqref{eq:resgamma}, $\Res_{s=0}\Gammas=1$, and therefore
\begin{equation}\label{eq:reszero}
\Res_{s=0}\ \frac{\Gamma(s)x^{-s}}{\zetas}=\lim_{s\to0}\frac{s\,\Gamma(s)\,x^{-s}}{\zetas}=\frac{1}{\zeta(0)}=-2 .
\end{equation}
Under (S), the residues at $0$ and in the critical strip therefore sum to
\begin{equation}\label{eq:critsum}
-2+\sum_{\rho}\frac{\Gamma(\rho)\,x^{-\rho}}{\zeta'(\rho)} ,
\end{equation}
with the summation convention \eqref{eq:symconv}.

\section{Calculation of the residues at the odd negative integers}\label{sec:odd}

First we recall the Bernoulli numbers, which are determined by the power series expansion
\begin{equation}\label{eq:berndef}
\frac{t}{e^{t}-1}=\sum_{n=0}^{\infty}\frac{B_{n}t^{n}}{n!}\qquad(\abs{t}<2\pi),
\end{equation}
so that $B_{2n+1}=0$ for $n\geq1$. Using the classical contour integral representation of $\zetas$ and the residue theorem (see \cite[Chapter II]{Titchmarsh} or \cite{Ayoub}) one obtains Euler's evaluation
\begin{equation}\label{eq:euler2m}
\zeta(2m)=(-1)^{m+1}\frac{(2\pi)^{2m}B_{2m}}{2\,(2m)!}\qquad(m\geq1),
\end{equation}
or equivalently
\begin{equation}\label{eq:berninv}
B_{2m+2}=(-1)^{m}\frac{2\,(2m+2)!\,\zeta(2m+2)}{(2\pi)^{2m+2}} .
\end{equation}
The analytic continuation of $\zetas$ to $\Complex\setminus\{1\}$ satisfies the functional equation (see \cite[Chapter II]{Titchmarsh})
\begin{equation}\label{eq:funceq}
\zetas=2^{s}\pi^{s-1}\sin\big(\tfrac{1}{2}\pi s\big)\,\Gamma(1-s)\,\zeta(1-s) .
\end{equation}
Equation \eqref{eq:funceq} implies that $\zetas$ has simple zeros at the points $s=-2,-4,-6,\dots$, the trivial zeros, and it also allows us to calculate the values of $\zetas$ at the odd negative integers. Namely,
\begin{equation}\label{eq:zetaoddval0}
\zeta\big(-(2m+1)\big)=2^{-(2m+1)}\pi^{-2m-2}\sin\big(-\tfrac{1}{2}\pi+\pi m\big)\,\Gamma(2m+2)\,\zeta(2m+2)\, .
\end{equation}
Here $\sin\big(\pi m-\tfrac{\pi}{2}\big)=-\cos(\pi m)=(-1)^{m+1}$ and $\Gamma(2m+2)=(2m+1)!$, so that
\begin{equation}\label{eq:zetaoddval1}
\zeta\big(-(2m+1)\big)=(-1)^{m+1}\,\frac{(2m+1)!\,\zeta(2m+2)}{2^{2m+1}\,\pi^{2m+2}} ,
\end{equation}
and substituting \eqref{eq:berninv} gives
\begin{equation}\label{eq:zetaoddval}
\zeta\big(-(2m+1)\big)=-\frac{B_{2m+2}}{2m+2}\qquad(m\geq0) .
\end{equation}
Now, since the pole of $\Gammas$ at $-(2m+1)$ is simple with residue $-1/(2m+1)!$ by \eqref{eq:resgamma}, while $\zetas$ is regular and nonzero there by \eqref{eq:zetaoddval}, we calculate
\begin{align}
&\Res_{s=-(2m+1)}\ \frac{\Gamma(s)x^{-s}}{\zetas}
=\lim_{s\to-(2m+1)}\big(s+2m+1\big)\,\frac{\Gamma(s)\,x^{-s}}{\zetas}\notag\\
&\qquad=\frac{(-1)^{2m+1}}{(2m+1)!}\cdot\frac{x^{2m+1}}{\zeta\big(-(2m+1)\big)}
=-\,\frac{x^{2m+1}}{(2m+1)!}\cdot\frac{2m+2}{-B_{2m+2}}\notag\\
&\qquad=\frac{(2m+2)\,x^{2m+1}}{(2m+1)!\,B_{2m+2}} .\label{eq:oddres}
\end{align}
Thus we have
\begin{equation}\label{eq:oddsum1}
\sum_{m=0}^{\infty}\Res_{s=-(2m+1)}\frac{\Gamma(s)x^{-s}}{\zetas}=\sum_{m=0}^{\infty}\frac{(2m+2)\,x^{2m+1}}{(2m+1)!\,B_{2m+2}} ,
\end{equation}
and using \eqref{eq:berninv} we obtain
\begin{equation}\label{eq:oddsum2}
\sum_{m=0}^{\infty}\Res_{s=-(2m+1)}\frac{\Gamma(s)x^{-s}}{\zetas}
=\pi\sum_{m=0}^{\infty}\frac{(-1)^{m}(2\pi x)^{2m+1}}{\zeta(2m+2)\,[(2m+1)!]^{2}} .
\end{equation}
Here we recognize the function $\kappa(x)$ of \eqref{eq:defkappabeta}, which admits the two equivalent representations
\begin{equation}\label{eq:kappatwo}
\kappa(x)=\sum_{m=0}^{\infty}\frac{(-1)^{m}(2\pi x)^{2m+1}}{\zeta(2m+2)\,[(2m+1)!]^{2}}
=\sum_{m=0}^{\infty}\frac{(2m+2)\,x^{2m+1}}{\pi\,(2m+1)!\,B_{2m+2}} .
\end{equation}
Its closed form as a M\"obius--Bessel series is given in Theorem \ref{thm:bessel} and proved in Section \ref{sec:special}. We conclude this section with the main result:
\begin{equation}\label{eq:oddfinal}
\sum_{m=0}^{\infty}\Res_{s=-(2m+1)}\frac{\Gamma(s)x^{-s}}{\zetas}=\pi\,\kappa(x) .
\end{equation}

\section{Calculation of the residues at the even negative integers}\label{sec:even}

At the negative even integers $\zetas$ has simple zeros and $\Gammas$ has simple poles, so the function $\Gamma(s)x^{-s}/\zetas$ has poles of order $2$ at the negative even integers. Therefore
\begin{equation}\label{eq:evenresdef}
\Res_{s=-2m}\ \frac{\Gamma(s)x^{-s}}{\zetas}=\lim_{s\to-2m}\dds\!\left(\frac{(s+2m)^{2}\,\Gammas\,x^{-s}}{\zetas}\right) .
\end{equation}
Now, using \eqref{eq:reflection} and \eqref{eq:funceq}, we can write
\begin{equation}\label{eq:evensplit}
\frac{(s+2m)^{2}\,\Gammas\,x^{-s}}{\zetas}
=\frac{(s+2m)^{2}\,x^{-s}\,\pi}{2^{s}\pi^{s-1}\,\Gamma(1-s)^{2}\,\zeta(1-s)\,\sin\big(\tfrac{1}{2}\pi s\big)\sin(\pi s)} .
\end{equation}
Let us denote
\begin{equation}\label{eq:fdef}
f(s)=\frac{\pi\,x^{-s}}{2^{s}\pi^{s-1}\,\Gamma(1-s)^{2}\,\zeta(1-s)}
=\frac{\pi^{2}}{(2\pi x)^{s}\,\Gamma(1-s)^{2}\,\zeta(1-s)} .
\end{equation}
Then for $\sigma<0$ the function $f(s)$ is analytic with no zeros and no poles, and we have the limit
\begin{multline}\label{eq:productrule}
\lim_{s\to-2m}\dds\!\left(f(s)\,\frac{(s+2m)^{2}}{\sin(\tfrac{1}{2}\pi s)\sin(\pi s)}\right)\\
=\lim_{s\to-2m}\left(f'(s)\,\frac{(s+2m)^{2}}{\sin(\tfrac{1}{2}\pi s)\sin(\pi s)}
+f(s)\,\dds\frac{(s+2m)^{2}}{\sin(\tfrac{1}{2}\pi s)\sin(\pi s)}\right).
\end{multline}
As is easily seen,
\begin{equation}\label{eq:limsin1}
\lim_{s\to-2m}\frac{s+2m}{\sin(\pi s)}=\frac{1}{\pi\cos(2\pi m)}=\frac{1}{\pi} ,
\end{equation}
and
\begin{equation}\label{eq:limsin2}
\lim_{s\to-2m}\frac{s+2m}{\sin(\tfrac{1}{2}\pi s)}=\frac{1}{\tfrac{1}{2}\pi\cos(\pi m)}=(-1)^{m}\,\frac{2}{\pi} ,
\end{equation}
so that, multiplying \eqref{eq:limsin1} and \eqref{eq:limsin2},
\begin{equation}\label{eq:limsinprod}
\lim_{s\to-2m}\frac{(s+2m)^{2}}{\sin(\tfrac{1}{2}\pi s)\sin(\pi s)}=\frac{2\,(-1)^{m}}{\pi^{2}} .
\end{equation}
Next we differentiate $f$. Since $f$ is nonvanishing on the simply connected half-plane $\operatorname{Re}s<0$, compatible holomorphic logarithms may be chosen there, with real values on the negative real axis. Taking logarithms in \eqref{eq:fdef},
\begin{equation}\label{eq:logf}
\log f(s)=2\log\pi-s\log(2\pi x)-2\log\Gamma(1-s)-\log\zeta(1-s) ,
\end{equation}
and therefore
\begin{equation}\label{eq:fprime}
f'(s)=f(s)\left(\frac{\zeta'(1-s)}{\zeta(1-s)}+\frac{2\,\Gamma'(1-s)}{\Gamma(1-s)}-\log(2\pi x)\right),
\end{equation}
since \[\dds\big(-s\log(2\pi x)\big)=-\log(2\pi x),\] \[\dds\big(-2\log\Gamma(1-s)\big)=2\psi(1-s)\] and \[\dds\big(-\log\zeta(1-s)\big)=\zeta'(1-s)/\zeta(1-s).\] Substituting $s=-2m$ we obtain
\begin{multline}\label{eq:fprimeval}
f'(-2m)=\frac{\pi^{2}\,(2\pi x)^{2m}}{[(2m)!]^{2}\,\zeta(2m+1)}\\
\times\left(\frac{\zeta'(2m+1)}{\zeta(2m+1)}+\frac{2\,\Gamma'(2m+1)}{(2m)!}-\log(2\pi x)\right).
\end{multline}
The second term in \eqref{eq:productrule} vanishes because
\begin{equation}\label{eq:secondzero}
\lim_{s\to-2m}\dds\,\frac{(s+2m)^{2}}{\sin(\tfrac{1}{2}\pi s)\sin(\pi s)}=0 .
\end{equation}
Indeed, writing $w=s+2m$, both $\sin(\pi s)=\sin(\pi w)$ and \[\sin(\tfrac{1}{2}\pi s)=(-1)^{m}\sin(\tfrac{1}{2}\pi w)\] are odd functions of $w$, so their product is an even function of $w$ with a double zero at $w=0$; hence there are coefficients $c_{0}\neq0,c_{2},c_{4},\dots$ such that near $w=0$
\begin{equation}\label{eq:evenexp}
\frac{w^{2}}{\sin(\tfrac{1}{2}\pi s)\sin(\pi s)}=c_{0}+c_{2}w^{2}+c_{4}w^{4}+\cdots ,
\end{equation}
an even analytic function of $w$ whose derivative vanishes at $w=0$. This proves \eqref{eq:secondzero}.

Combining \eqref{eq:productrule}, \eqref{eq:limsinprod}, \eqref{eq:fprimeval} and \eqref{eq:secondzero}, we obtain the residue at the double pole:
\begin{multline}\label{eq:evenres}
\Res_{s=-2m}\ \frac{\Gamma(s)x^{-s}}{\zetas}
=\frac{2\,(-1)^{m}(2\pi x)^{2m}}{[(2m)!]^{2}\,\zeta(2m+1)}\\
\times\left(\frac{\zeta'(2m+1)}{\zeta(2m+1)}+\frac{2\,\Gamma'(2m+1)}{(2m)!}-\log(2\pi x)\right).
\end{multline}
An instructive consistency check, the direct Laurent computation of the same residue through the expansions of $\Gamma$ and of $1/\zeta$ at\\ $s=-2m$ together with the functional-equation evaluations of $\zeta'(-2m)$ and $\zeta''(-2m)$, is carried out in full in Appendices \ref{app:gammalaurent}--\ref{app:sanity} and reproduces \eqref{eq:evenres} exactly.

We now recall the functions $\lambda(x)$ and $\upsilon(x)$ of \eqref{eq:deflamups} and the function $\beta(x)$ of \eqref{eq:defkappabeta}. Summing \eqref{eq:evenres} over $m\geq1$ and separating the three terms in the bracket, we conclude this section with the final result:
\begin{equation}\label{eq:evenfinal}
\sum_{m=1}^{\infty}\Res_{s=-2m}\frac{\Gamma(s)x^{-s}}{\zetas}
=2\lambda(x)+4\upsilon(x)-2\beta(x)\log(2\pi x) .
\end{equation}
Finally, combining \eqref{eq:critsum}, \eqref{eq:oddfinal} and \eqref{eq:evenfinal}: if all the nontrivial zeros of the Riemann zeta function are simple, then the sum of the residues of $\Gamma(s)x^{-s}/\zetas$ over all poles with $\operatorname{Re}s<\tfrac32$ is
\begin{equation}\label{eq:allres}
\sum_{\rho}\frac{\Gamma(\rho)\,x^{-\rho}}{\zeta'(\rho)}
+\pi\kappa(x)+2\lambda(x)+4\upsilon(x)-2\beta(x)\log(2\pi x)-2 .
\end{equation}

\section{The integral}\label{sec:integral}

We shift the line in \eqref{eq:mellininv} to the left. Constants below may depend on fixed $x>0$ and $\sigma_0>1$, but not on $N$ or $t$.

\subsection{The reflected representation}\label{subsec:reflected}
First let us observe that, by \eqref{eq:reflection},
\begin{equation}\label{eq:gammarefl}
\Gamma(s)=\frac{\pi}{\sin(\pi s)\,\Gamma(1-s)} ,
\end{equation}
and by the functional equation \eqref{eq:funceq},
\begin{equation}\label{eq:zetarefl}
\frac{1}{\zeta(s)}=\frac{1}{2^{s}\pi^{s-1}\sin(\tfrac{1}{2}\pi s)\,\Gamma(1-s)\,\zeta(1-s)} .
\end{equation}
This implies, for the integrand $F(s)=\Gamma(s)x^{-s}/\zetas$,
\begin{equation}\label{eq:Frefl}
F(s)=\frac{\pi^{2}\,(2\pi x)^{-s}}{\sin(\pi s)\,\sin(\tfrac{1}{2}\pi s)\,\Gamma(1-s)^{2}\,\zeta(1-s)} .
\end{equation}
The representation \eqref{eq:Frefl} will be used on the part of the contour with $\sigma\le0$, where it converts all the quantities to be estimated into quantities evaluated in the half-plane $\operatorname{Re}(1-s)\geq1$.

The elementary identity
\begin{equation}\label{eq:sinsigit2}
 |\sin(\sigma+it)|^2=\sin^2\sigma+\sinh^2t
\end{equation}
gives the exact vertical-line estimates

\begin{align}
\abs{\sin\big(\pi(\tfrac12-2N+it)\big)}&=\cosh(\pi t)\geq\tfrac12 e^{\pi\abs{t}},\label{eq:sinexact}\\
\abs{\sin\big(\tfrac{\pi}{2}(\tfrac12-2N+it)\big)}&=\sqrt{\tfrac12\cosh(\pi t)}\geq\tfrac{1}{\sqrt2} ,\label{eq:sinexact2}
\end{align}
and, on horizontal lines $t=\pm T$ with $T\geq1$,
\begin{align}
\abs{\sin\big(\pi(\sigma\pm iT)\big)}&\geq\sinh(\pi T)\geq\tfrac14 e^{\pi T},\label{eq:sinhoriz}\\
\abs{\sin\big(\tfrac{\pi}{2}(\sigma\pm iT)\big)}&\geq\sinh\big(\tfrac{\pi T}{2}\big)\geq\tfrac14 e^{\pi T/2} .\label{eq:sinhoriz2}
\end{align}
Appendix~\ref{app:sines} records the derivations.

\subsection{Two lemmas}\label{subsec:lemmas}
The first lemma replaces the asymptotic estimate \eqref{eq:stirling}, which is valid for fixed $\sigma$ only, by a lower bound valid uniformly in $\sigma$; this uniformity is exactly what is needed on the moving line $\sigma=\half-2N$.

\begin{lem}[Uniform lower bound for $\Gamma$ on vertical lines]\label{lem:gammalower}
For every $\sigma>0$ and every real $t$,
\begin{equation}\label{eq:gammalower}
\abs{\Gamma(\sigma+it)}\;\geq\;\Gamma(\sigma)\left(1+\frac{t^{2}}{\sigma^{2}}\right)^{-1/2}e^{-\pi\abs{t}/2} .
\end{equation}
\end{lem}

\begin{proof}
The Weierstrass product \cite[Chapter~XII]{WW}, applied at $\sigma+it$ and $\sigma$, gives
\begin{equation}\label{eq:gammaproduct}
\frac{|\Gamma(\sigma+it)|^2}{\Gamma(\sigma)^2}
 =\prod_{n=0}^{\infty}\left(1+\frac{t^2}{(n+\sigma)^2}\right)^{-1}.
\end{equation}
Since $u\mapsto\log(1+t^2/u^2)$ is nonnegative and decreasing,
\begin{equation}\label{eq:sumint}
\sum_{n=0}^{\infty}\log\left(1+\frac{t^2}{(n+\sigma)^2}\right)
\leq\log\left(1+\frac{t^2}{\sigma^2}\right)
 +\int_0^\infty\log\left(1+\frac{t^2}{u^2}\right)du.
\end{equation}
Substitution and integration by parts evaluate the integral as $\pi|t|$. Exponentiation proves the bound. Appendix~\ref{app:gammaproduct} supplies the product calculation and the vanishing boundary terms.
\end{proof}

The second lemma provides horizontal lines that avoid the zeros of $\zetas$ quantitatively; it is classical.

\begin{lem}[Ordinate selection]\label{lem:ordinates}
There are an absolute constant $A_{0}>0$ and heights $T_{\nu}\in[\nu,\nu+1]$, $\nu=2,3,4,\dots$, such that
\begin{equation}\label{eq:ordbound}
\frac{1}{\abs{\zeta(\sigma\pm iT_{\nu})}}\;\leq\;T_{\nu}^{\,A_{0}}
\qquad(-1\leq\sigma\leq2).
\end{equation}
In particular $\zeta$ has no zero on the two lines $t=\pm T_{\nu}$, so no singularity of $F$ lies on them.
\end{lem}

\begin{rem}[The selected heights]\label{rem:ordsource}
Lemma~\ref{lem:ordinates} is \cite[Theorem~9.7]{Titchmarsh}; conjugation gives both signs. Its polynomial bound is used only on the near horizontal segments, where Gamma decay dominates it. The other segments use \eqref{eq:zetalow}, directly or after reflection.
\end{rem}

\subsection{Proof of Theorem \ref{thm:main}}\label{subsec:proofmain}
Fix $x>0$ and set $\sigma_{0}=\tfrac32$. For $N\geq2$ let $T_{N}$ be as in Lemma \ref{lem:ordinates} and let $R_{N}$ denote the positively oriented rectangle with vertices
\begin{equation}\label{eq:rectangle}
\sigma_{0}\pm iT_{N},\qquad\big(\tfrac12-2N\big)\pm iT_{N} .
\end{equation}
The boundary avoids all poles. The interior contains the nontrivial zeros with $|\gamma|<T_N$, the pole at $0$, the odd-negative poles $-(2m+1)$ for $0\leq m\leq N-1$, and the even-negative poles $-2m$ for $1\leq m\leq N-1$. At $s=1$ the integrand vanishes. With the orientation shown in Figure~\ref{fig:contour}, the residue theorem gives
\begin{equation}\label{eq:resthm}
\frac{1}{2\pi i}\oint_{R_{N}}F(s)\,ds
=\sum_{\abs{\gamma}<T_{N}}\frac{\Gamma(\rho)x^{-\rho}}{\zeta'(\rho)}-2
+\sum_{m=0}^{N-1}\Res_{-(2m+1)}F+\sum_{m=1}^{N-1}\Res_{-2m}F .
\end{equation}

We now show that, as $N\to\infty$, the right edge $\mathcal{C}_{1}$ of $R_{N}$ produces the integral \eqref{eq:mellininv} while the remaining three segments $\mathcal{C}_{2}$, $\mathcal{C}_{3}$ and $\mathcal{C}_{4}$ tend to zero.

\medskip
\noindent\emph{Segment $\mathcal C_1$: the right edge.}
On $\sigma=\sigma_0$, \eqref{eq:absint} proves absolute integrability. More explicitly, the omitted tail is bounded by a constant times
\[
 x^{-\sigma_0}\int_{T_N}^\infty e^{-\pi t/2}t^{\sigma_0-1/2}\,dt
 =x^{-\sigma_0}(2/\pi)^{\sigma_0+1/2}
   \Gamma(\sigma_0+\tfrac12,\tfrac\pi2T_N),
\]
which tends to zero. Here $\Gamma(a,z)=\int_z^\infty e^{-u}u^{a-1}\,du$ is the upper incomplete Gamma function \cite[Section~8.2]{DLMF}. Hence
\begin{equation}\label{eq:rightedge}
\frac{1}{2\pi i}\int_{\sigma_{0}-iT_{N}}^{\sigma_{0}+iT_{N}}F(s)\,ds
\longrightarrow\frac{1}{2\pi i}\int_{\sigma_{0}-i\infty}^{\sigma_{0}+i\infty}F(s)\,ds=\Phi(e^{-x})
\qquad(N\to\infty),
\end{equation}
by \eqref{eq:mellininv}.

\begin{figure}[t]
\centering
\begin{tikzpicture}[scale=0.86, >=stealth]
\fill[gray!12] (0,-4.4) rectangle (1,4.4);
\draw[->,thick] (-7.6,0) -- (2.6,0) node[right] {$\operatorname{Re}s$};
\draw[->,thick] (0,-4.9) -- (0,4.9) node[above] {$\operatorname{Im}s$};
\foreach \y/\lab in {4/T_{N},3.05/T_{N-1},2.35/T_{N-2}} {
  \draw[very thick] (-0.09,\y) -- (0.09,\y);
  \node[left=1pt, fill=white, inner sep=1.5pt] at (-0.12,\y) {$\lab$};
}
\foreach \y/\lab in {-4/-T_{N},-3.05/-T_{N-1}} {
  \draw[very thick] (-0.09,\y) -- (0.09,\y);
  \node[left=1pt, fill=white, inner sep=1.5pt] at (-0.12,\y) {$\lab$};
}
\node[left] at (-0.09,1.55) {$\vdots$};
\node[left] at (-0.09,-1.9) {$\vdots$};
\draw[contourRight, very thick, ->] (1.5,-4) -- node[right=2pt, pos=0.78] {$\mathcal{C}_{1}:\ \sigma=\sigma_{0}$} (1.5,4);
\draw[contourTop, very thick, ->] (1.5,4) -- node[above=2pt] {$\mathcal{C}_{2}:\ t=+T_{N}$} (-6.5,4);
\draw[contourLeft, very thick, ->] (-6.5,4) -- node[pos=0.27, right=2pt, rotate=90, anchor=south] {$\mathcal{C}_{3}:\ \sigma=\tfrac12-2N$} (-6.5,-4);
\draw[contourBottom, very thick, ->] (-6.5,-4) -- node[below=2pt] {$\mathcal{C}_{4}:\ t=-T_{N}$} (1.5,-4);
\foreach \y in {0.9,1.75,2.7,3.5,-0.9,-1.75,-2.7,-3.5} {
  \draw[thick] (0.5-0.09,\y-0.09) -- (0.5+0.09,\y+0.09);
  \draw[thick] (0.5+0.09,\y-0.09) -- (0.5-0.09,\y+0.09);
}
\node[right=1pt] at (0.6,2.7) {$\rho$};
\foreach \x in {0,-1,-3,-5} \fill (\x,0) circle (0.075);
\foreach \x in {-2,-4,-6} { \fill (\x,0) circle (0.075); \draw (\x,0) circle (0.16); }
\node[below=3pt] at (-1,0) {$-1$};
\node[below=3pt] at (-2,0) {$-2$};
\node[below=3pt] at (-4,0) {$-4$};
\node[below=3pt] at (-6,0) {$-6$};
\draw[very thick] (1.5,-0.09) -- (1.5,0.09);
\node[anchor=west] at (1.62,0.75) {$\sigma_{0}=\tfrac32$};
\draw (1,0) circle (0.075);
\node[below=3pt] at (1,0) {$1$};
\end{tikzpicture}
\caption{The positively oriented contour $R_N$ of \eqref{eq:rectangle}. The critical strip is shaded. Crosses indicate nontrivial zeros, drawn on the critical line only for illustration; RH is not assumed. Dots mark Gamma poles, and circled dots mark double poles where Gamma poles meet trivial zeta zeros. The open circle at $s=1$ is a regular point.}
\label{fig:contour}
\end{figure}
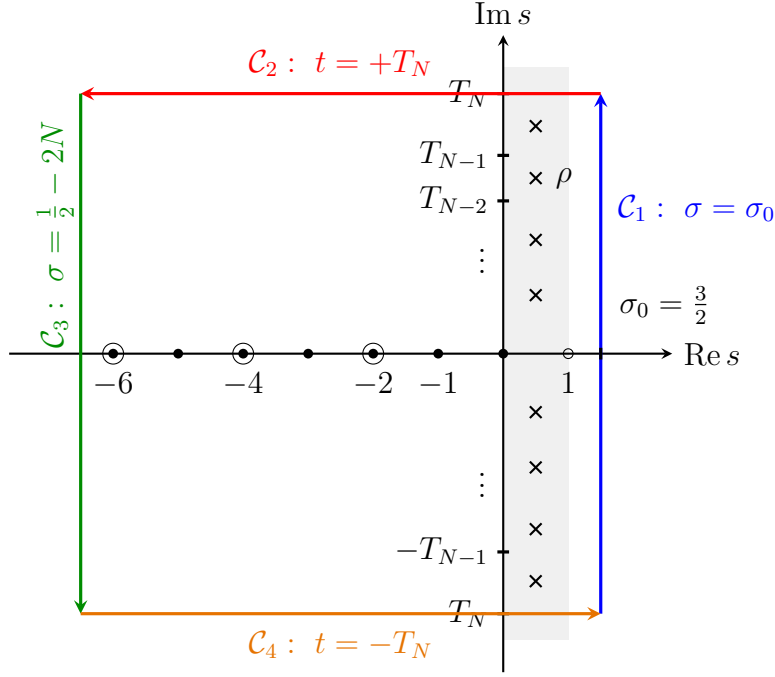

\medskip
\noindent\emph{Segment $\mathcal C_3$: the left edge.}
Write $s=\tfrac12-2N+it$, $|t|\leq T_N$, in \eqref{eq:Frefl}. Then $|(2\pi x)^{-s}|=(2\pi x)^{2N-1/2}$ and $\operatorname{Re}(1-s)=2N+\tfrac12\geq\tfrac52$. Lemma~\ref{lem:gammalower} gives
\begin{equation}\label{eq:gammaleft}
\abs{\Gamma(1-s)}\geq\Gamma\big(2N+\tfrac12\big)\big(1+t^{2}\big)^{-1/2}e^{-\pi\abs{t}/2} ,
\end{equation}
using $1+t^2/(1-\sigma)^2\leq1+t^2$. Also $|\zeta(1-s)|\geq\zeta(5)/\zeta(5/2)=:c_0>0$. With \eqref{eq:sinexact}--\eqref{eq:sinexact2}, these bounds yield
\begin{align}
\abs{F(s)}
&\leq\frac{\pi^{2}\,(2\pi x)^{2N-\frac12}\cdot 2e^{-\pi\abs{t}}\cdot\sqrt2\cdot(1+t^{2})\,e^{\pi\abs{t}}}{\Gamma(2N+\tfrac12)^{2}\,c_{0}}\label{eq:leftbound}\\
&=C_{2}\,\frac{(2\pi x)^{2N}\,(1+t^{2})}{\Gamma(2N+\tfrac12)^{2}} .\notag
\end{align}
The exponential factors cancel. Since $T_N\leq N+1$, integration gives
\begin{equation}\label{eq:leftedge}
\abs{\int_{\text{left edge}}F(s)\,ds}
\leq C_{3}\,\frac{(2\pi x)^{2N}\,N^{3}}{\Gamma(2N+\tfrac12)^{2}}\longrightarrow0\qquad(N\to\infty),
\end{equation}
superexponentially, for every fixed $x>0$, since the factorial in the denominator dominates any fixed exponential.

\medskip
\noindent\emph{Segments $\mathcal C_2$ and $\mathcal C_4$: the horizontal edges.}
On $t=\pm T_N$, split the integral at $\sigma=-1$. For $-1\leq\sigma\leq\sigma_0$, the ordinate bound gives $|\zeta(s)|^{-1}\leq T_N^{A_0}$. Stirling's formula, uniformly on this fixed strip, gives
\begin{equation}\label{eq:stirlingcompact}
\abs{\Gamma(\sigma\pm iT)}\leq C_{4}\,T^{\sigma_{0}-\frac12}\,e^{-\pi T/2}
\qquad(-1\leq\sigma\leq\sigma_{0},\ T\geq1),
\end{equation}
together with $x^{-\sigma}\leq\max(x^{-3/2},x)$, yields
\begin{equation}\label{eq:horiznear}
\abs{\int_{-1}^{\sigma_{0}}F(\sigma\pm iT_{N})\,d\sigma}
\leq C_{5}(x)\,T_{N}^{\,A_{0}+\sigma_{0}-\frac12}\,e^{-\pi T_{N}/2}\longrightarrow0\qquad(N\to\infty),
\end{equation}
the exponential factor $e^{-\pi T_{N}/2}$ dominating the fixed power of $T_{N}$.

For $\tfrac12-2N\leq\sigma\leq-1$ we use \eqref{eq:Frefl} again. By \eqref{eq:sinhoriz} and \eqref{eq:sinhoriz2}, by Lemma \ref{lem:gammalower} applied at $\operatorname{Re}(1-s)=1-\sigma\geq2$,
\begin{equation}\label{eq:gammahoriz}
\abs{\Gamma(1-s)}\geq\Gamma(1-\sigma)\,(1+T_{N}^{2})^{-1/2}\,e^{-\pi T_{N}/2},
\end{equation}
and $|\zeta(1-s)|\geq\zeta(4)/\zeta(2)$. Thus
\begin{align}
\abs{F(\sigma\pm iT_{N})}
&\leq C_{6}\;\frac{(2\pi x)^{-\sigma}\,(1+T_{N}^{2})}{\Gamma(1-\sigma)^{2}}\,
e^{-\pi T_{N}-\frac{\pi T_{N}}{2}+\pi T_{N}}\label{eq:horizfar1}\\
&=C_{6}\,\frac{(2\pi x)^{-\sigma}\,(1+T_{N}^{2})}{\Gamma(1-\sigma)^{2}}\,e^{-\pi T_{N}/2} ,\notag
\end{align}
To integrate, enlarge the interval to $\sigma\leq-1$ and substitute $u=1-\sigma$:

\begin{equation}\label{eq:Ax}
W(x):=\int_{-\infty}^{-1}\frac{(2\pi x)^{-\sigma}}{\Gamma(1-\sigma)^{2}}\,d\sigma
=\int_{2}^{\infty}\frac{(2\pi x)^{u-1}}{\Gamma(u)^{2}}\,du\;<\;\infty ,
\end{equation}
The integral is finite by Stirling's formula on the positive real axis. Therefore
\begin{equation}\label{eq:horizfar2}
\abs{\int_{\frac12-2N}^{-1}F(\sigma\pm iT_{N})\,d\sigma}
\leq C_{6}\,W(x)\,(1+T_{N}^{2})\,e^{-\pi T_{N}/2}\longrightarrow0\quad(N\to\infty).
\end{equation}

\medskip
\begin{proof}[Completion of the contour argument]
By \eqref{eq:rightedge}, \eqref{eq:leftedge}, \eqref{eq:horiznear}, and \eqref{eq:horizfar2}, the contour integral in \eqref{eq:resthm} tends to $\Phi(e^{-x})$. The trivial-pole series converge absolutely by their squared-factorial denominators (Lemma~\ref{lem:entire}); Sections~\ref{sec:odd}--\ref{sec:even} identify their sums. The remaining term, the symmetric zero sum, must consequently converge as well. Rearranging proves \eqref{eq:mainformula}. The coefficient calculation in Theorem~\ref{thm:unified} then proves \eqref{eq:explicit-unified}, completing Theorem~\ref{thm:main}.
\end{proof}

\begin{rem}[Multiple zeros]\label{rem:mult}
Hypothesis (S) enters only through the closed form \eqref{eq:resrho}. Without it, each distinct zero $\rho$ of multiplicity $m_{\rho}$ contributes once:
\begin{equation}\label{eq:multres}
\frac{1}{(m_{\rho}-1)!}\;\lim_{s\to\rho}\frac{d^{\,m_{\rho}-1}}{ds^{\,m_{\rho}-1}}
\left((s-\rho)^{m_{\rho}}\,\frac{\Gamma(s)\,x^{-s}}{\zeta(s)}\right),
\end{equation}
and Theorem \ref{thm:main} holds unconditionally in this form, exactly as in the unconditional treatment of the Mertens formula by Bartz \cite{Bartz1991}.
\end{rem}

\subsection{A quantitative refinement on the same contour}
On the half-integer lines used in the preceding proof, a second estimate
gives an explicit remainder and makes the dependence on $x$ uniform
on compact subintervals of $(0,\infty)$.

\begin{prop}[Half-integer estimate and left-line remainder]\label{prop:leftremainder}
For $N\geq1$ and $t\in\mathbb R$,
\begin{equation}\label{eq:halfintgamma}
 \abs{\Gamma(2N+\tfrac12-it)}
 \geq\frac{\Gamma(2N+\tfrac12)}{\sqrt{\cosh(\pi t)}}.
\end{equation}
Consequently, for $x>0$,
\begin{multline}\label{eq:leftremainder}
\abs{\frac{1}{2\pi i}\int_{\frac12-2N-i\infty}^{\frac12-2N+i\infty}
 \frac{\Gamma(s)x^{-s}}{\zeta(s)}\,ds}
 \leq C_*\frac{(2\pi x)^{2N-\frac12}}{\Gamma(2N+\frac12)^2},\\
 C_* =\sqrt{\frac\pi2}\,
 \frac{\zeta(5/2)}{\zeta(5)}\,
 \frac{\Gamma(1/4)}{\Gamma(3/4)}.
\end{multline}
The symmetric zero sum in Theorem~\ref{thm:main} converges locally
uniformly for $x>0$.
\end{prop}

\begin{proof}
The reflection identity and the recurrence for $\Gamma$ give
\begin{align*}
 \abs{\Gamma(\tfrac12-it)}^2&=\frac{\pi}{\cosh(\pi t)},\\
 \abs{\Gamma(2N+\tfrac12-it)}
 &=\abs{\Gamma(\tfrac12-it)}
       \prod_{j=0}^{2N-1}\abs{j+\tfrac12-it}\\
 &\geq\sqrt{\frac{\pi}{\cosh(\pi t)}}
       \prod_{j=0}^{2N-1}(j+\tfrac12)
 =\frac{\Gamma(2N+\tfrac12)}{\sqrt{\cosh(\pi t)}}.
\end{align*}
On $s=\frac12-2N+it$, the two sine factors in the reflected
representation have product
$2^{-1/2}\cosh(\pi t)^{3/2}$. Also
$\abs{\zeta(1-s)}\geq\zeta(5)/\zeta(5/2)$, by
\eqref{eq:zetalow}. Thus
\[
 \abs{\frac{\Gamma(s)x^{-s}}{\zeta(s)}}
 \leq\pi^2\sqrt2\,\frac{\zeta(5/2)}{\zeta(5)}
 \frac{(2\pi x)^{2N-\frac12}}{\Gamma(2N+\frac12)^2}
 \frac1{\sqrt{\cosh(\pi t)}}.
\]
With $v=\tanh(\pi t)$, a beta integral evaluates the remaining integral:
\[
 \int_{-\infty}^{\infty}\frac{dt}{\sqrt{\cosh(\pi t)}}
 =\frac1\pi B(\tfrac12,\tfrac14)
 =\frac{\Gamma(1/4)}{\sqrt\pi\,\Gamma(3/4)}.
\]
This proves \eqref{eq:leftremainder}. The bounds on the other three
edges in the contour proof are uniform when $x$ belongs to a
fixed compact subinterval of $(0,\infty)$: the factors $x^{-\sigma}$
are then bounded by the corresponding endpoint powers, and the
superexponential Gamma majorant controls the far horizontal edges.
The power series for the trivial residues converge uniformly on such
compact sets. Solving the finite contour identity for the zero sum
therefore proves the stated local uniform convergence.
\end{proof}

\begin{rem}
For fixed $N$, shifting first to the entire line
$\operatorname{Re}s=\frac12-2N$ and then letting the selected horizontal
height tend to infinity identifies its integral with the part of the
trivial-pole contribution beyond that line. Hence
\eqref{eq:leftremainder} is a bound for truncating the \emph{trivial}
residues. It is not a bound for truncating the series over nontrivial
zeros. The distinction is essential when the explicit formula is used
numerically.
\end{rem}

\section{\texorpdfstring{The special functions $\kappa,\beta,\lambda,\upsilon$}{The special functions kappa, beta, lambda, upsilon}}\label{sec:special}

\begin{lem}\label{lem:entire}
The four power series defining $\kappa(x)$, $\beta(x)$, $\lambda(x)$ and $\upsilon(x)$ have infinite radius of convergence; the four functions are entire.
\end{lem}

\begin{proof}
Writing any of the four series as $\sum_n a_nx^n$, we claim that
\begin{equation}\label{eq:coeffbound}
 |a_n|\leq C\frac{(n+1)(2\pi)^n}{(n!)^2}.
\end{equation}
The reciprocal zeta factors are at most $1$, and
$|\zeta'(2m+1)|\leq\sum_{k\geq2}(\log k)/k^3$ for $m\geq1$.
For the derivative factor in $\upsilon$, differentiation under the Gamma integral gives, for real $s>1$,
\begin{multline}\label{eq:gammaprimebound}
 |\Gamma'(s)|\leq\int_0^\infty t^{s-1}|\log t|e^{-t}\,dt\\
 \leq\int_0^1t^{s-1}(-\log t)\,dt+
          \int_1^\infty t^se^{-t}\,dt
 \leq s^{-2}+\Gamma(s+1).
\end{multline}
Thus $|\Gamma'(2m+1)|/[(2m)!]^3\leq2(2m+1)/[(2m)!]^2$, which proves \eqref{eq:coeffbound}. Since $n!\geq(n/e)^n$,
\[
 |a_n|^{1/n}\leq\frac{2\pi e^2[C(n+1)]^{1/n}}{n^2}\longrightarrow0.
\]
The Cauchy--Hadamard formula proves the assertion.
\end{proof}

We now prove Theorem \ref{thm:bessel}. The Bessel function of the first kind of order zero is
\begin{equation}\label{eq:J0def}
J_{0}(z)=\sum_{n=0}^{\infty}\frac{(-1)^{n}}{(n!)^{2}}\Big(\frac{z}{2}\Big)^{2n}
\end{equation}
(see \cite[Chapter 5]{Lebedev} or \cite[Section 10.2]{DLMF}). Substituting the rotated argument $z=2e^{-i\pi/4}\sqrt{y}$, for which $(z/2)^{2}=e^{-i\pi/2}\,y=-iy$, we obtain the identity
\begin{equation}\label{eq:J0rot}
J_{0}\big(2e^{-i\pi/4}\sqrt{y}\big)=\sum_{n=0}^{\infty}\frac{(-1)^{n}(-i)^{n}}{(n!)^{2}}\,y^{n}
=\sum_{n=0}^{\infty}\frac{i^{n}}{(n!)^{2}}\,y^{n}\qquad(y\ge0),
\end{equation}
whose real and imaginary parts are precisely the two lacunary series that arise from the trivial zeros:
\begin{align}
\operatorname{Re}J_{0}\big(2e^{-i\pi/4}\sqrt{y}\big)&=\sum_{m=0}^{\infty}\frac{(-1)^{m}y^{2m}}{[(2m)!]^{2}},\label{eq:J0parts}\\
\operatorname{Im}J_{0}\big(2e^{-i\pi/4}\sqrt{y}\big)&=\sum_{m=0}^{\infty}\frac{(-1)^{m}y^{2m+1}}{[(2m+1)!]^{2}} .\label{eq:J0parts2}
\end{align}
(By the classical relation $\ber(z)+i\,\bei(z)=J_{0}\big(z\,e^{3\pi i/4}\big)$ and the evenness of $J_{0}$, the real and imaginary parts in \eqref{eq:J0parts} and \eqref{eq:J0parts2} are precisely the Kelvin functions $\ber(2\sqrt{y})$ and $\bei(2\sqrt{y})$; see \cite[Section 10.61]{DLMF}.)

\begin{proof}[Proof of Theorem \ref{thm:bessel}]
Expanding $1/\zeta(2m+2)=\sum_{k\ge1}\mu(k)k^{-2m-2}$ in the definition \eqref{eq:defkappabeta} of $\kappa$ and interchanging the order of summation,
\begin{align}
\kappa(x)&=\sum_{k=1}^{\infty}\frac{\mu(k)}{k}\sum_{m=0}^{\infty}\frac{(-1)^{m}}{[(2m+1)!]^{2}}\Big(\frac{2\pi x}{k}\Big)^{2m+1}\label{eq:kappaswap}\\
&=\sum_{k=1}^{\infty}\frac{\mu(k)}{k}\,\operatorname{Im}J_{0}\Big(2e^{-i\pi/4}\sqrt{\tfrac{2\pi x}{k}}\Big),\notag
\end{align}
by \eqref{eq:J0parts2} with $y=2\pi x/k$. The interchange is justified by Tonelli's theorem: putting $y_{k}=2\pi x/k$ and $g(y)=\sum_{m\geq0}y^{2m+1}/[(2m+1)!]^{2}$, one has $g(y)\leq y\,e^{y}$, so
\begin{equation}\label{eq:tonelli}
\sum_{k=1}^{\infty}\frac{1}{k}\,g(y_{k})
\leq\sum_{k\leq2\pi x}\frac{g(y_{k})}{k}+\sum_{k>2\pi x}\frac{y_{k}e^{y_{k}}}{k}
\leq C(x)+2\pi e\,x\sum_{k>2\pi x}\frac{1}{k^{2}}<\infty .
\end{equation}
For $\beta$, put $g_e(y)=\sum_{m\geq1}y^{2m}/[(2m)!]^2$. The bound
$g_e(y)\leq\cosh y-1\leq\tfrac12y^2e^y$ gives
\[
 \sum_{k\geq1}\frac{g_e(2\pi x/k)}k
 \leq C(x)+\frac{e(2\pi x)^2}{2}\sum_{k>2\pi x}k^{-3}<\infty.
\]
Expanding $1/\zeta(2m+1)$ and interchanging absolutely therefore yields

\begin{align}
\beta(x)&=\sum_{k=1}^{\infty}\frac{\mu(k)}{k}\sum_{m=1}^{\infty}\frac{(-1)^{m}}{[(2m)!]^{2}}\Big(\frac{2\pi x}{k}\Big)^{2m}\label{eq:betaswap}\\
&=\sum_{k=1}^{\infty}\frac{\mu(k)}{k}\left(\operatorname{Re}J_{0}\Big(2e^{-i\pi/4}\sqrt{\tfrac{2\pi x}{k}}\Big)-1\right),\notag
\end{align}
The subtraction of $1$ removes the constant term. Adding $i$ times \eqref{eq:kappaswap} to \eqref{eq:betaswap} gives \eqref{eq:besselcombined}; no summation of the unregularized series $\sum_k\mu(k)/k$ is needed.

For $\lambda$ we use the Dirichlet series of the derivative of $1/\zeta$: for $\operatorname{Re}s>1$,
\begin{equation}\label{eq:dirichletprime}
\frac{\zeta'(s)}{\zeta(s)^{2}}=-\Big(\frac{1}{\zeta}\Big)'(s)=\sum_{k=1}^{\infty}\frac{\mu(k)\log k}{k^{s}} ,
\end{equation}
obtained by differentiating \eqref{eq:dirichlet} term by term. Substituting \eqref{eq:dirichletprime} at $s=2m+1$ into the definition \eqref{eq:deflamups} of $\lambda$ and interchanging, with the absolute majorant now $\sum_{k}(\log k/k)\,g_{e}(y_{k})<\infty$, whose tail beyond $k=2\pi x$ is bounded by $\tfrac{e}{2}(2\pi x)^{2}\sum_{k}(\log k)\,k^{-3}$ by the same estimate, we obtain
\begin{equation}\label{eq:lambdaswap}
\lambda(x)=\sum_{k=1}^{\infty}\frac{\mu(k)\log k}{k}\left(\operatorname{Re}J_{0}\Big(2e^{-i\pi/4}\sqrt{\tfrac{2\pi x}{k}}\Big)-1\right),
\end{equation}
which proves \eqref{eq:bessellambda}. The subtraction again removes the large-$k$ constant term; the factor $\log k$ does not make it dispensable.
\end{proof}

\subsection{The fourth function and a unified M\"obius--Bessel formula}\label{subsec:unified}

We now derive a representation for $\upsilon$ and combine the four terms occurring in the explicit formula. Throughout, the functions have the definitions \eqref{eq:defkappabeta}--\eqref{eq:deflamups}, and $\gamma_e$ denotes Euler's constant.

\begin{prop}[An integral identity for $\upsilon$]\label{prop:upsilon}
For every $x\in\Complex$,
\begin{equation}\label{eq:upsilon-beta}
\upsilon(x)=-\gamma_e\beta(x)+\int_0^1\frac{\beta(x)-\beta(tx)}{1-t}\,dt.
\end{equation}
The integrand has the removable value $x\beta'(x)$ at $t=1$. For $x>0$ this gives the absolutely convergent M\"obius-weighted representation
\begin{multline}\label{eq:upsilon-mobius}
\upsilon(x)=\sum_{k=1}^{\infty}\frac{\mu(k)}{k}\operatorname{Re}\Bigg\{
-\gamma_e\left[J_0\left(2e^{-i\pi/4}\sqrt{\frac{2\pi x}{k}}\right)-1\right]\\
+\int_0^1\frac{J_0\left(2e^{-i\pi/4}\sqrt{\frac{2\pi x}{k}}\right)
-J_0\left(2e^{-i\pi/4}\sqrt{\frac{2\pi tx}{k}}\right)}{1-t}\,dt\Bigg\}.
\end{multline}
\end{prop}

\begin{proof}
Use $\Gamma'(2m+1)/(2m)!=\psi(2m+1)=H_{2m}-\gamma_e$ and
\begin{equation}\label{eq:harmonic-integral}
 H_{2m}=\int_0^1\frac{1-t^{2m}}{1-t}\,dt.
\end{equation}
Substitution in \eqref{eq:deflamups} gives
\[
 \upsilon(x)+\gamma_e\beta(x)
 =\int_0^1\sum_{m\geq1}
 \frac{(-1)^m(2\pi x)^{2m}(1-t^{2m})}
      {[(2m)!]^2\zeta(2m+1)(1-t)}\,dt,
\]
which is \eqref{eq:upsilon-beta}. The interchange is absolute and locally uniform: on $|x|\leq R$, the integrated majorant is
\[
 \sum_{m\geq1}\frac{(2\pi R)^{2m}H_{2m}}
 {[(2m)!]^2\zeta(2m+1)}<\infty,
\]
since $H_{2m}\leq1+\log(2m)$. Taylor expansion of $\beta(tx)$ gives the removable value at $t=1$.

To obtain \eqref{eq:upsilon-mobius}, substitute the M\"obius--Bessel series for $\beta$. Absolute interchange with the $t$-integral follows from
\[
 \sum_{k\geq1}\frac1k\sum_{m\geq1}
 \frac{(2\pi x/k)^{2m}(H_{2m}+\gamma_e)}{[(2m)!]^2}<\infty.
\]
For $k>2\pi x$ the inner sum is $O_x(k^{-2})$, so the outer tail is summable.
\end{proof}

The next identity combines the four functions into one M\"obius-weighted series. Here $K_0$ is the modified Bessel function of order zero on its principal branch. The four-function combination accounts for the negative-integer poles; the residue at $s=0$ remains the separate constant $-2$.

\begin{thm}[A unified M\"obius--Bessel formula]\label{thm:unified}
For every $x>0$,
\begin{multline}\label{eq:unified-k0}
\pi\kappa(x)+2\lambda(x)+4\upsilon(x)-2\beta(x)\log(2\pi x)\\
=\sum_{k=1}^{\infty}\frac{\mu(k)}k
\left[4\operatorname{Re}K_0\left(2e^{i\pi/4}\sqrt{\frac{2\pi x}{k}}\right)
+2\log\left(\frac{2\pi x}{k}\right)+4\gamma_e\right].
\end{multline}
The series converges absolutely and gives the second form \eqref{eq:explicit-unified} of Theorem~\ref{thm:main}.
\end{thm}

\begin{proof}
The coefficient calculation identifies each of the four contributions explicitly. Temporarily put $y=2\pi x/k$. The power-series expansion of $K_0$ (\cite[Section~10.31]{DLMF}) gives
\begin{multline}\label{eq:k0-series-original-j0}
K_0(2e^{i\pi/4}\sqrt y)
=-\left(\tfrac12\log y+\tfrac{i\pi}{4}\right)J_0(2e^{-i\pi/4}\sqrt y)\\
+\sum_{j=0}^{\infty}\frac{\psi(j+1)i^j y^j}{(j!)^2}.
\end{multline}
Indeed, the modified-Bessel power series contains $(iy)^j/(j!)^2$, which is exactly the rotated $J_0$ series \eqref{eq:J0rot}. Taking real parts, separating the $j=0$ term $\psi(1)=-\gamma_e$, and multiplying by $4$, we obtain
\begin{align}\label{eq:kernel-expanded}
&4\operatorname{Re}K_0(2e^{i\pi/4}\sqrt y)+2\log y+4\gamma_e\notag\\
&\quad=\pi\sum_{m=0}^{\infty}\frac{(-1)^m y^{2m+1}}{[(2m+1)!]^2}
+4\sum_{m=1}^{\infty}\frac{(-1)^m\psi(2m+1)y^{2m}}{[(2m)!]^2}\notag\\
&\qquad\quad-2\log y\sum_{m=1}^{\infty}\frac{(-1)^m y^{2m}}{[(2m)!]^2}.
\end{align}
Multiply this identity by $\mu(k)/k$ and sum over $k$. The first series gives $\pi\kappa(x)$; the second gives $4\upsilon(x)$. In the last series,
\[
\log y=\log(2\pi x)-\log k,
\]
so its two contributions are $-2\beta(x)\log(2\pi x)$ and $2\lambda(x)$, respectively. These are precisely the four terms on the left of \eqref{eq:unified-k0}, with the indicated signs and coefficients.

For absolute convergence, \eqref{eq:kernel-expanded} yields, as $y\to0^{+}$,
\begin{multline}\label{eq:kernel-small}
4\operatorname{Re}K_0(2e^{i\pi/4}\sqrt y)+2\log y+4\gamma_e\\
=\pi y+y^2\bigl(\tfrac12\log y-\psi(3)\bigr)+O(y^3+y^4|\log y|).
\end{multline}
The $k$th weighted summand is therefore $O_x(k^{-2})$, and the logarithmic part is $O_x((1+\log k)k^{-3})$. The previous absolute-convergence arguments also justify the separated coefficient computations. The three terms \emph{inside the square brackets} in \eqref{eq:unified-k0} must be kept together when the $k$-sum is taken: it is their cancellation at $y=0$, not separate convergence of unregularized logarithmic sums, that proves the result.
\end{proof}

\begin{prop}[A positive integral and a truncation estimate]\label{prop:positive-kernel}
For $y>0$ the bracket in \eqref{eq:unified-k0} has the representation
\begin{multline}\label{eq:positive-kernel}
4\operatorname{Re}K_0(2e^{i\pi/4}\sqrt y)+2\log y+4\gamma_e\\
=4\int_0^\infty(1-e^{-y/v})\frac{\sin^2(v/2)}v\,dv,
\end{multline}
and it lies strictly between $0$ and $\pi y$. For every integer $K\ge1$,
\begin{multline}\label{eq:kernel-tail}
\left|\sum_{k>K}\frac{\mu(k)}k
\left[4\operatorname{Re}K_0\left(2e^{i\pi/4}\sqrt{\frac{2\pi x}{k}}\right)
+2\log\left(\frac{2\pi x}{k}\right)+4\gamma_e\right]\right|\\
\le\frac{2\pi^2x}{K}.
\end{multline}
In particular,
\begin{equation}\label{eq:correction-global}
\left|\pi\kappa(x)+2\lambda(x)+4\upsilon(x)-2\beta(x)\log(2\pi x)\right|\le30x.
\end{equation}
\end{prop}

\begin{proof}
The regularization in \eqref{eq:positive-kernel} can be verified without splitting divergent integrals. For $\varepsilon>0$, use the standard integral
\[
\int_0^\infty e^{-av-y/v}\frac{dv}{v}=2K_0(2\sqrt{ay})
\qquad(\operatorname{Re}a>0,\ y>0)
\]
from \cite[Section~10.32]{DLMF}, together with Frullani's integral, to obtain
\begin{align*}
&2\int_0^\infty e^{-\varepsilon v}(1-e^{-y/v})\frac{1-\cos v}{v}\,dv\\
&\quad=\log(1+\varepsilon^{-2})-4K_0(2\sqrt{\varepsilon y})
+4\operatorname{Re}K_0(2\sqrt{(\varepsilon-i)y}).
\end{align*}
Each difference used here is integrable. As $\varepsilon\to0^{+}$, the left side increases to the integral in \eqref{eq:positive-kernel}. On the right,
\[
K_0(2\sqrt{\varepsilon y})=-\tfrac12\log(\varepsilon y)-\gamma_e+o(1),
\]
and complex conjugation identifies the real part at $\sqrt{-iy}$ with the real part at $\sqrt{iy}$. The two divergent $\log\varepsilon$ terms cancel, leaving exactly \eqref{eq:positive-kernel}.

Positivity is immediate. The inequality $0<1-e^{-y/v}<y/v$ gives
\[
0<4\int_0^\infty(1-e^{-y/v})\frac{\sin^2(v/2)}v\,dv
<4y\int_0^\infty\frac{\sin^2(v/2)}{v^2}\,dv=\pi y.
\]
For completeness, integration by parts reduces the last integral to
$\tfrac12\int_0^\infty\sin v\,dv/v=\pi/4$.
After putting $y=2\pi x/k$, comparison with $\sum_{k>K}k^{-2}\le1/K$ proves \eqref{eq:kernel-tail}. Finally,
\[
\sum_{k=1}^{\infty}\frac{|\mu(k)|}{k^2}=\prod_p(1+p^{-2})=\frac{\zeta(2)}{\zeta(4)}=\frac{15}{\pi^2},
\]
so the same estimate for the full series gives $2\pi^2x\cdot15/\pi^2=30x$. Positivity belongs to the individual bracket, not to the M\"obius-weighted sum.
\end{proof}

\begin{rem}[Relation to the general M\"obius summation formula]\label{rem:chorge-dixit}
Taking the Mellin data $\Gamma(s)x^{-s}$, inverse transform $e^{-xt}$, and residue $k=1$ in \cite[Theorem~2.9 of the arXiv version]{ChorgeDixit} gives
\[
\Phi(e^{-x})=\sum_\rho\frac{\Gamma(\rho)x^{-\rho}}{\zeta'(\rho)}-2
-4\sum_{k\ge1}\frac{\mu(k)}k\int_0^\infty(e^{-xt}-1)\frac{\sin^2(\pi/(kt))}{t}\,dt.
\]
The substitution $v=2\pi/(kt)$ identifies its integral with \eqref{eq:positive-kernel}, at $y=2\pi x/k$. Thus the underlying real simple-zero summation identity is contained in that general theorem. The computations above give its explicit evaluation in the four functions of this paper, complete the representation of $\upsilon$, and exhibit the cancellations leading to one regularized $K_0$ series.
\end{rem}

\section{The sum over the zeros}\label{sec:zerosum}

\begin{thm}\label{thm:zerosum}
Assume Hypothesis (S). For every $0<x<\infty$ the symmetric limit
\begin{equation}\label{eq:zerosumlim}
\sum_{\rho}\frac{\Gamma(\rho)\,x^{-\rho}}{\zeta'(\rho)}
=\lim_{\nu\to\infty}\sum_{\abs{\gamma}<T_{\nu}}\frac{\Gamma(\rho)\,x^{-\rho}}{\zeta'(\rho)}
\end{equation}
exists and is a real number.
\end{thm}

\begin{proof}
Convergence follows from the contour argument. Conjugation preserves the zeros and their derivatives, and $x>0$ gives
\begin{equation}\label{eq:pairreal}
\frac{\Gamma(\bar\rho)\,x^{-\bar\rho}}{\zeta'(\bar\rho)}=\overline{\left(\frac{\Gamma(\rho)\,x^{-\rho}}{\zeta'(\rho)}\right)} ,
\end{equation}
Each finite symmetric sum therefore equals twice the real part of the sum over its positive ordinates. Its limit is real, with no rearrangement of an infinite series.
\end{proof}

For an absolutely convergent version we introduce the following hypothesis.

\medskip
\noindent\textbf{Hypothesis (H).} $\displaystyle\sum_{\rho}\frac{\abs{\Gamma(\rho)}}{\abs{\zeta'(\rho)}}<\infty .$
\medskip

\noindent Since, by \eqref{eq:stirling} with the constant uniform for $\beta_{\rho}\in[0,1]$,
\[\abs{\Gamma(\rho)}\;\ll\;e^{-\pi\abs{\gamma}/2}\,\big(1+\abs{\gamma}\big)^{\beta_{\rho}-\frac12} ,\]
and the number of zeros with $\abs{\gamma}\leq T$ is $O(T\log T)$ (\cite[Chapter IX]{Titchmarsh}), Hypothesis (H) holds as soon as, for a single fixed $\theta<\pi/2$,
\[\frac{1}{\abs{\zeta'(\rho)}}\;\ll\;e^{\theta\abs{\gamma}}\qquad\text{over the nontrivial zeros}.\] A polynomial upper bound for a global negative moment is sufficient as well. Under (S), if
\[
 J_{-1}(T):=\sum_{0<\gamma\leq T}\frac1{\abs{\zeta'(\rho)}^2}
       \ll (1+T)^B
\]
for some fixed $B\geq0$, positivity of the summands implies
\[
 \frac1{\abs{\zeta'(\rho)}^2}
 \leq J_{-1}(\abs\gamma+1)\ll(1+\abs\gamma)^B.
\]
Conjugation gives the same bound for negative ordinates, and Gamma decay proves (H). Thus global upper bounds for the nonnegative sums defining negative moments studied in \cite{Gonek1989,Hejhal1989} imply suitable pointwise control. Simplicity alone supplies no such uniform estimate.

The functional equation pairs $\rho$ with $1-\rho$, one in each half of the critical strip. All following pairings take place within each finite symmetric cutoff, so no rearrangement is involved:

\begin{equation}\label{eq:pairsplit}
\sum_{\rho}\frac{\Gamma(\rho)\,x^{-\rho}}{\zeta'(\rho)}
=\lim_{\nu\to\infty}\ \sum_{0<\operatorname{Im}\rho<T_{\nu}}\left(\frac{\Gamma(\rho)\,x^{-\rho}}{\zeta'(\rho)}+\frac{\Gamma(1-\rho)\,x^{-(1-\rho)}}{\zeta'(1-\rho)}\right).
\end{equation}
Differentiating \eqref{eq:funceq} at $s=\rho$ eliminates the term containing $\zeta(1-\rho)$ and gives
\begin{equation}\label{eq:zprho2}
\zeta'(\rho)=-\,2^{\rho}\pi^{\rho-1}\sin(\tfrac12\pi\rho)\,\Gamma(1-\rho)\,\zeta'(1-\rho) ,
\end{equation}
and hence
\begin{equation}\label{eq:zprho3}
\zeta'(1-\rho)=-\,\frac{\zeta'(\rho)\,2^{-\rho}\,\pi^{1-\rho}}{\Gamma(1-\rho)\,\sin(\tfrac12\pi\rho)} .
\end{equation}
Substituting \eqref{eq:zprho3}, and keeping careful track of the minus sign,
\begin{align}
\frac{\Gamma(1-\rho)\,x^{-(1-\rho)}}{\zeta'(1-\rho)}
&=-\,\frac{\Gamma(1-\rho)^{2}\,\sin(\tfrac12\pi\rho)\,2^{\rho}\pi^{\rho-1}}{\zeta'(\rho)}\;x^{\rho-1}\label{eq:pairterm}\\
&=-\,\pi\,\frac{(2\pi)^{\rho}\,\sin(\tfrac12\pi\rho)}{\zeta'(\rho)\,\Gamma(\rho)^{2}\,\sin^{2}(\pi\rho)}\;x^{\rho-1},\notag
\end{align}
where in the last step we used $\Gamma(1-\rho)^{2}=\pi^{2}/\big(\Gamma(\rho)^{2}\sin^{2}(\pi\rho)\big)$ from \eqref{eq:gammarefl} and $2\pi^{2}\,2^{\rho-1}\pi^{\rho-1}=\pi(2\pi)^{\rho}$. Therefore
\begin{multline}\label{eq:pairident1}
\sum_{\rho}\frac{\Gamma(\rho)\,x^{-\rho}}{\zeta'(\rho)}\\
=\lim_{\nu\to\infty}\ \sum_{0<\operatorname{Im}\rho<T_{\nu}}\left[\frac{\Gamma(\rho)}{\zeta'(\rho)}\,x^{-\rho}
-\,\frac{2\,(2\pi)^{\rho-1}\,\Gamma(1-\rho)^{2}\,\sin(\tfrac12\pi\rho)}{\zeta'(\rho)}\;x^{\rho-1}\right].
\end{multline}
\subsection{A uniform condition allowing cancellation}
Hypothesis (H) bounds the zero sum by the sum of the absolute values
of its coefficients. To retain cancellation, consider instead finite
Fej\'er-weighted sums.

\medskip
\noindent\textbf{Hypothesis $(H_{\mathrm F})$.} Under RH and (S), there
is a constant $C$ such that, simultaneously for every $T>0$ and $x>0$,
\begin{equation}\label{eq:HF}
 \abs{\sum_{\abs{\gamma}<T}
    \left(1-\frac{\abs{\gamma}}T\right)
    \frac{\Gamma(\rho)}{\zeta'(\rho)}x^{-i\gamma}}
 \leq C.
\end{equation}
Every sum in \eqref{eq:HF} is finite. In particular, the condition makes
sense without (H) and without a choice of the heights $T_\nu$.

\begin{thm}[A cancellation-preserving endpoint characterization]\label{thm:fejer}
Assume RH and (S). Then \eqref{eq:rhbound} is equivalent to
$(H_{\mathrm F})$. Hypothesis (H) implies $(H_{\mathrm F})$. Moreover,
if the endpoint bound holds, then
\begin{equation}\label{eq:fejernorm}
 \sup_{T>0,\,x>0}\abs{\sum_{\abs{\gamma}<T}
       \left(1-\frac{\abs{\gamma}}T\right)
       \frac{\Gamma(\rho)}{\zeta'(\rho)}x^{-i\gamma}}
 =\limsup_{x\to0^{+}}\sqrt{x}\abs{\Phi(e^{-x})}.
\end{equation}
No independence hypothesis on the zero ordinates is needed.
\end{thm}

\begin{proof}
Within this proof write
\[
 a_\rho=\frac{\Gamma(\rho)}{\zeta'(\rho)},\qquad
 L=\limsup_{x\to0^{+}}\sqrt{x}\abs{\Phi(e^{-x})}.
\]
For $u\in\mathbb R$, put $x=e^{-u}$ and define
\begin{multline}\label{eq:proofh}
 h(u)=e^{-u/2}\big\{\Phi(e^{-x})-\pi\kappa(x)-2\lambda(x)
       -4\upsilon(x)\\
       +2\beta(x)\log(2\pi x)+2\big\},\qquad x=e^{-u}.
\end{multline}
Theorem~\ref{thm:main} and Proposition~\ref{prop:leftremainder} give
\begin{equation}\label{eq:hgrouped}
 h(u)=\lim_{\nu\to\infty}\sum_{\abs{\gamma}<T_\nu}
                  a_\rho e^{i\gamma u},
\end{equation}
locally uniformly in $u$. In particular, $h$ is continuous. Also, from
the defining power series,
\begin{equation}\label{eq:hlimsup}
 \limsup_{u\to\infty}\abs{h(u)}=L.
\end{equation}
This equality is valid with either side possibly infinite.

\smallskip
\noindent\emph{A test-function identity.}
Whenever $\sup_\rho\abs{a_\rho}<\infty$, local uniform convergence
and $N(T)=O(T\log T)$ imply, for every smooth compactly supported
function $\varphi$,
\begin{equation}\label{eq:htest}
 \int_{\mathbb R}h(u)\varphi(u)\,du
       =\sum_\rho a_\rho\widehat\varphi(-\gamma),
 \qquad
 \widehat\varphi(t)=\int_{\mathbb R}\varphi(u)e^{-itu}\,du.
\end{equation}
Indeed, integrate the finite sums in \eqref{eq:hgrouped} first.
Repeated integration by parts gives
$\abs{\widehat\varphi(t)}\leq C_{\varphi,A}(1+\abs t)^{-A}$ for
every integer $A$, so the series on the right is absolutely convergent.
This proves \eqref{eq:htest} without rearranging the pointwise zero sum.

\smallskip
\noindent\emph{$(H_{\mathrm F})$ implies the endpoint bound.}
Write $P_T(u)$ for the finite sum in \eqref{eq:HF} with $x=e^{-u}$.
For a fixed zero ordinate $\gamma_0$ and $T>\abs{\gamma_0}$,
\[
 \lim_{U\to\infty}\frac1{2U}\int_{-U}^{U}
          P_T(u)e^{-i\gamma_0u}\,du
     =\left(1-\frac{\abs{\gamma_0}}T\right)a_{\rho_0}.
\]
Distinct ordinates correspond to distinct zeros under RH and (S).
Thus the bound $\abs{P_T}\leq C$ gives $\abs{a_\rho}\leq C$ on
letting $T\to\infty$. For a compactly supported smooth $\varphi$,
absolute convergence of the series in \eqref{eq:htest} now yields
\[
 \lim_{T\to\infty}\int_{\mathbb R}P_T(u)\varphi(u)\,du
       =\int_{\mathbb R}h(u)\varphi(u)\,du.
\]
The absolute value of the latter integral is at most
$C\int\abs\varphi$. Approximation by nonnegative smooth functions
concentrated near any prescribed point, and continuity of $h$, give
$\abs{h(u)}\leq C$ everywhere. Equation~\eqref{eq:hlimsup} proves the
endpoint bound.

\smallskip
\noindent\emph{The endpoint bound implies a global bound for $h$.}
Assume now $L<\infty$. We first obtain bounded coefficients directly
from the Mellin transform, not from absolute convergence of the zero
series. The argument in Lemma~\ref{lem:bound-holomorphy} and in the proof of
Theorem~\ref{thm:criterion}(a) gives $G(s)=\Gamma(s)/\zeta(s)$ for
$\operatorname{Re}s>1/2$. For each fixed $\rho=1/2+i\gamma$,
\[
 a_\rho=\lim_{\eta\to0^{+}}\eta G(\rho+\eta).
\]
Given $\epsilon>0$, choose $\delta>0$ such that
$\sqrt{x}\abs{\Phi(e^{-x})}\leq L+\epsilon$ for $0<x<\delta$.
The part of the Mellin integral over $[\delta,\infty)$ remains bounded
as $\eta\to0^{+}$, whereas its part over $(0,\delta)$ has modulus
at most $(L+\epsilon)\delta^\eta/\eta$. Therefore
\begin{equation}\label{eq:acoefbound}
 \abs{a_\rho}\leq L\qquad\text{for every }\rho.
\end{equation}
The identity \eqref{eq:htest} consequently applies.

There is a sequence $\tau_j\to\infty$ such that
$e^{i\gamma\tau_j}\to1$ for every zero ordinate $\gamma$.
Indeed, for any finite set of frequencies, infinitely many vectors $(e^{i\gamma_1n},\ldots,e^{i\gamma_rn})$ lie in one small box of the compact torus. Differences of arbitrarily separated indices give arbitrarily large simultaneous return times. Apply this with accuracy $1/j$ to the first $j$ positive ordinates and choose $\tau_j>j$. Negative ordinates follow by conjugation.

Apply \eqref{eq:htest} to a translate of $\varphi$. Absolute convergence
and \eqref{eq:acoefbound} give
\[
 \int_{\mathbb R}h(u+\tau_j)\varphi(u)\,du
   =\sum_\rho a_\rho e^{i\gamma\tau_j}\widehat\varphi(-\gamma)
   \longrightarrow\int_{\mathbb R}h(u)\varphi(u)\,du.
\]
On the support of the translated test function, \eqref{eq:hlimsup}
bounds $\abs{h}$ by $L+\epsilon$ for all sufficiently large $j$.
It follows that $\abs{\int h\varphi}\leq L\int\abs\varphi$.
Continuity again implies
\begin{equation}\label{eq:hglobal}
 \sup_{u\in\mathbb R}\abs{h(u)}=L.
\end{equation}
The equality, rather than just the upper bound, follows from
\eqref{eq:hlimsup}.

\smallskip
\noindent\emph{Positive convolution gives the finite Fej\'er sums.}
The classical real-line Fej\'er kernel is
\begin{equation}\label{eq:fejerkernel}
 K_T(v)=\frac{1-\cos(Tv)}{\pi T v^2},\quad
 K_T(0)=\frac{T}{2\pi},\qquad
 \int_{\mathbb R}K_T(v)\,dv=1.
\end{equation}
It is nonnegative and has Fourier transform
$\widehat K_T(t)=(1-\abs t/T)_+$, where $v_+=\max(v,0)$; see
\cite{TitchmarshFI}. The normalization follows from
$\int_{\mathbb R}(1-\cos v)v^{-2}\,dv=\pi$.
The test-function identity implies
\begin{equation}\label{eq:fejerconvolution}
 \int_{\mathbb R}h(u-v)K_T(v)\,dv
       =\sum_{\abs{\gamma}<T}
          \left(1-\frac{\abs\gamma}{T}\right)a_\rho e^{i\gamma u}
       =P_T(u).
\end{equation}
Here is a justification of this use of a kernel that is not compactly
supported. Because $h$ is bounded and the coefficients in
\eqref{eq:htest} are bounded, a smooth cutoff first extends that
identity to Schwartz test functions. Apply it to
$K_T(v)e^{-\epsilon v^2}$, translated by $u$. Its Fourier transform is
the convolution of $(1-\abs t/T)_+$ with the unit-mass Gaussian
$(4\pi\epsilon)^{-1/2}\exp(-t^2/(4\epsilon))$. For $\abs\gamma>T+1$ the transforms are bounded, uniformly
for $0<\epsilon\leq1$, by a constant depending on $T$ times
$e^{-(\abs\gamma-T)^2/8}$. The zero count therefore permits dominated
convergence in the coefficient sum. The finitely many remaining
ordinates converge to the triangular weights, including weight zero
at an endpoint. Dominated convergence in the integral follows from
$\abs{h}\leq L$ and $K_T\in L^1$. This proves
\eqref{eq:fejerconvolution}.

Positivity and unit mass now give $\abs{P_T(u)}\leq L$ for all $T,u$.
Conversely, $K_T$ is an approximate identity as $T\to\infty$, so
$P_T(u)\to h(u)$ at every $u$. Equation~\eqref{eq:hglobal} proves
\eqref{eq:fejernorm}. Finally, (H) implies \eqref{eq:HF} immediately
by the triangle inequality.
\end{proof}

\begin{rem}[Scope of the condition]\label{rem:HFscope}
The implication $(H)\Rightarrow(H_{\mathrm F})$ permits cancellation instead of requiring absolute summability. Bounded Fej\'er means do not imply absolute summability for general Fourier series, but strictness of this implication for the zeta spectrum is not established here. In particular, the uniform estimate \eqref{eq:HF} remains an additional condition, not a consequence proved from RH and (S).
\end{rem}

\section{A criterion for the Riemann hypothesis and simplicity}\label{sec:criterion}

\begin{proof}[Proof of Theorem \ref{thm:criterion}(b)]
Part (b) follows from Theorem~\ref{thm:fejer}. Under (H), the direct argument uses $\rho=\tfrac12+i\gamma$ in \eqref{eq:mainformula}:
\begin{multline}\label{eq:rhforward1}
\sum_{n=1}^{\infty}\mu(n)e^{-nx}
=x^{-\frac12}\sum_{\rho}\frac{\Gamma(\rho)}{\zeta'(\rho)}\,x^{-i\gamma}\\
+\pi\kappa(x)+2\lambda(x)+4\upsilon(x)-2\beta(x)\log(2\pi x)-2 .
\end{multline}
Under (H),
\begin{equation}\label{eq:rhforward2}
\abs{\sum_\rho\frac{\Gamma(\rho)}{\zeta'(\rho)}x^{-i\gamma}}
\leq\sum_\rho\abs{\frac{\Gamma(\rho)}{\zeta'(\rho)}}<\infty
\end{equation}
uniformly in $x$. The remaining contribution is $-2+O(x)$ by \eqref{eq:correction-global}, proving \eqref{eq:rhbound}.
\end{proof}

\begin{lem}[Holomorphy forced by the decay bound]\label{lem:bound-holomorphy}
Assume that
\[
\Phi(e^{-x})=\sum_{n=1}^{\infty}\mu(n)e^{-nx}=O(x^{-1/2})
\qquad(x\to0^{+}).
\]
Define
\begin{equation}\label{eq:Gdef}
G(s):=\int_{0}^{\infty}\Phi(e^{-x})x^{s-1}\,dx,
\qquad x^{s-1}:=e^{(s-1)\log x}\quad(x>0).
\end{equation}
The integral in \eqref{eq:Gdef} is absolutely convergent and holomorphic in the half-plane
\[
\Omega:=\set{s\in\Complex:\operatorname{Re}(s)>\tfrac12}.
\]
\end{lem}

\begin{proof}
Extend the assumed bound to $0<x\leq1$ by increasing its constant. For $\sigma=\operatorname{Re}s>1/2$,
\begin{equation}\label{eq:G-near-zero}
 \int_0^1|\Phi(e^{-x})x^{s-1}|\,dx
 \leq C\int_0^1x^{\sigma-3/2}\,dx=\frac{C}{\sigma-1/2}.
\end{equation}
For $x\geq1$, $|\Phi(e^{-x})|\leq(e^x-1)^{-1}<2e^{-x}$, so
\begin{equation}\label{eq:G-tail}
 \int_1^\infty|\Phi(e^{-x})x^{s-1}|\,dx
 \leq2\int_1^\infty e^{-x}x^{\sigma-1}\,dx<\infty.
\end{equation}
The majorants are uniform when $s$ ranges over a compact subset of $\Omega$; the same is true after multiplication by any fixed power of $|\log x|$. Differentiation under the integral proves holomorphy \cite[Section~2.5(i)]{DLMF}.
\end{proof}

\begin{proof}[Proof of Theorem \ref{thm:criterion}(a)]
Let $G$ be the holomorphic function on $\Omega=\{\operatorname{Re}s>1/2\}$ furnished by Lemma~\ref{lem:bound-holomorphy} under the assumed bound \eqref{eq:rhbound}. Initially, \eqref{eq:mellinfwd} gives
\begin{equation}\label{eq:product-identity-initial}
 \zeta(s)G(s)=\Gamma(s)\qquad(\operatorname{Re}s>1).
\end{equation}
On the connected domain $D=\Omega\setminus\{1\}$, the function
\begin{equation}\label{eq:F-product}
 F(s)=\zeta(s)G(s)-\Gamma(s)
\end{equation}
is holomorphic, including at any hypothetical zero of $\zeta$. The identity theorem therefore gives
\begin{equation}\label{eq:zeta-G-Gamma}
 \zeta(s)G(s)=\Gamma(s)\qquad(s\in D).
\end{equation}
A zero $\rho\in D$ would imply $\Gamma(\rho)=0$, which is impossible. Thus
\begin{equation}\label{eq:no-zeros-right}
 \zeta(s)\neq0\qquad(\operatorname{Re}s>1/2,\ s\neq1),
\end{equation}
where $s=1$ is its pole, not a zero. Reflection by the functional equation excludes nontrivial zeros to the left of the critical line, proving RH. This argument applies the identity theorem to a product; it does not assume holomorphy of $1/\zeta$ in the larger half-plane.

It remains to prove simplicity. Fix a zero $\rho=1/2+i\gamma$. For $0<\eta\leq1/4$, the same Mellin estimates give
\begin{equation}\label{eq:radial-mellin-bound}
 |G(\rho+\eta)|\leq C\int_0^1x^{\eta-1}\,dx+O_\rho(1)
 =\frac C\eta+O_\rho(1).
\end{equation}
If $\rho$ has multiplicity $m_\rho$, its Taylor expansion and \eqref{eq:zeta-G-Gamma} yield
\begin{equation}\label{eq:radial-pole-order}
 G(\rho+\eta)=
 \frac{m_\rho!\,\Gamma(\rho)}{\zeta^{(m_\rho)}(\rho)}
 \eta^{-m_\rho}(1+O_\rho(\eta))\qquad(\eta\to0^{+}).
\end{equation}
The leading coefficient is nonzero. Comparison with \eqref{eq:radial-mellin-bound} forces $m_\rho=1$. Neither (S), (H), nor $(H_{\mathrm F})$ has been assumed.
\end{proof}

\subsection{Multiplicity information and necessary residue bounds}

\begin{prop}[Logarithmic losses measure possible multiplicity]\label{prop:multlog}
Let $A\geq0$. If
\begin{equation}\label{eq:logcriterion}
 \Phi(e^{-x})=O\!\left(x^{-1/2}(\log(1/x))^A\right)
       \qquad(x\to0^{+}),
\end{equation}
then RH holds and every nontrivial zero has multiplicity at most $A+1$.
If, for an integer $k\geq0$, the right-hand side is replaced by
$o(x^{-1/2}(\log(1/x))^k)$, every multiplicity is at most $k$.
In particular, the displayed big-$O$ estimate with $0\leq A<1$
implies RH and (S). The weaker little-$o$ estimate
$o(x^{-1/2}\log(1/x))$ also implies RH and (S), whereas
$o(x^{-1/2})$ is impossible.
\end{prop}

\begin{proof}
For every $\sigma>1/2$, a logarithmic power is integrable against
$x^{\sigma-3/2}$ near zero. The holomorphy and product-identity argument above therefore still prove RH.
With $u=\log(1/x)$,
\begin{equation}\label{eq:logmellin}
 \int_0^1x^{\eta-1}(\log(1/x))^A\,dx
    =\int_0^\infty e^{-\eta u}u^A\,du
    =\frac{\Gamma(A+1)}{\eta^{A+1}}.
\end{equation}
The part of $G(\rho+\eta)$ outside a sufficiently small fixed interval
near zero is bounded as $\eta\to0^{+}$. The assumed bound and
\eqref{eq:logmellin} thus give
$G(\rho+\eta)=O_\rho(\eta^{-A-1})$. Equation
\eqref{eq:radial-pole-order} forces $m_\rho\leq A+1$.
For a little-$o$ estimate, first make the constant in the bound near
zero arbitrarily small, then let $\eta\to0^{+}$. This gives
$G(\rho+\eta)=o_\rho(\eta^{-k-1})$, excluding multiplicity $k+1$ as
well as larger multiplicities. For $k=0$ this contradicts the existence
of nontrivial zeros.
\end{proof}

\begin{prop}[Square-summability forced by the endpoint]\label{prop:residueL2}
If \eqref{eq:rhbound} holds and
$L=\limsup_{x\to0^{+}}\sqrt{x}\abs{\Phi(e^{-x})}$, then
\begin{equation}\label{eq:residueL2}
 \sum_\rho\abs{\frac{\Gamma(\rho)}{\zeta'(\rho)}}^2\leq L^2.
\end{equation}
In particular $L>0$ and
\begin{equation}\label{eq:zprime-lower}
 \abs{\zeta'(\tfrac12+i\gamma)}
 \geq\frac1L\sqrt{\frac{\pi}{\cosh(\pi\gamma)}}.
\end{equation}
No hypothesis on the convergence of a zero sum is required.
\end{prop}

\begin{proof}
By Theorem~\ref{thm:criterion}(a), RH and (S) have already been proved.
For $u\geq0$, write only within this proof
$f(u)=e^{-u/2}\Phi(e^{-e^{-u}})$ and
$a_\rho=\Gamma(\rho)/\zeta'(\rho)$. The function $f$ is bounded and
$\limsup_{u\to\infty}\abs{f(u)}=L$. Substitution $x=e^{-u}$ gives
\begin{multline}\label{eq:Abelcoefficients}
 \eta\int_0^\infty e^{-\eta u}f(u)e^{-i\tau u}\,du
  =\eta\int_0^1\Phi(e^{-x})x^{-1/2+\eta+i\tau}\,dx\\
 \longrightarrow
 \begin{cases}
 a_\rho,&\rho=\frac12+i\tau\text{ is a zero},\\
 0,&\frac12+i\tau\text{ is not a zero}.
 \end{cases}
\end{multline}
Indeed the omitted part of $G$ over $[1,\infty)$ is analytic and is
annihilated in the limit by the factor $\eta$, and the limit of
$\eta G(1/2+\eta+i\tau)$ is its simple-pole residue or zero.
For finitely many zeros let $P(u)=\sum_{j=1}^r a_{\rho_j}e^{i\gamma_j u}$.
Expand the nonnegative integral
\[
 0\leq\eta\int_0^\infty e^{-\eta u}\abs{f(u)-P(u)}^2\,du.
\]
The mixed term tends to $\sum_j\abs{a_{\rho_j}}^2$ by
\eqref{eq:Abelcoefficients}, and the $\abs P^2$ term has the same
limit because
$\eta\int_0^\infty e^{-\eta u}e^{i(\gamma_j-\gamma_k)u}\,du$
tends to $1$ for $j=k$ and to $0$ otherwise. The limsup of the $\abs f^2$
term is at most $L^2$: split at a point beyond which
$\abs f\leq L+\epsilon$, then let $\eta\to0^{+}$ and
$\epsilon\to0^{+}$. Thus
$\sum_{j=1}^r\abs{a_{\rho_j}}^2\leq L^2$. Taking the supremum over
finite sets proves \eqref{eq:residueL2}. Each coefficient is nonzero,
so $L>0$. The identity
$\abs{\Gamma(1/2+i\gamma)}^2=\pi/\cosh(\pi\gamma)$ yields
\eqref{eq:zprime-lower}.
\end{proof}

\begin{rem}
The derivative bound \eqref{eq:zprime-lower} decays exponentially in $|\gamma|$. Neither it nor \eqref{eq:residueL2} proves (H). Simplicity alone gives only pointwise nonvanishing of the derivative; zeros escaping to infinity need not have derivatives bounded away from zero.
\end{rem}

\subsection{Relation to the convolution criterion}
For clarity, the endpoint implication in B\'aez-Duarte's Theorem~4.1 \cite{BaezDuarte} specializes exactly to this transform. In his notation, take
\[
 g(v)=\sum_{n\leq v}\frac{\mu(n)}n,\qquad
 h(u)=u^{-1}e^{-1/u},\qquad
 \varphi(u)=u h'(u).
\]
Then $\varphi(u)=(u^{-2}-u^{-1})e^{-1/u}$ and
\begin{align*}
 G_\varphi(t)&=\int_0^\infty g(t/u)\varphi(u)\frac{du}{u}\\
 &=\sum_{n\geq1}\frac{\mu(n)}n\int_0^{t/n}h'(u)\,du
 =\frac1t\Phi(e^{-1/t}).
\end{align*}
The interchange is absolute: for $n>t$, $h'$ is positive on $(0,t/n)$ and its integral is $h(t/n)$; the other $n$ form a finite set. Moreover,
\[
 \varphi^\wedge(s)=\int_0^\infty\varphi(u)u^{-s-1}\,du
 =\Gamma(s+2)-\Gamma(s+1)=s\Gamma(s+1).
\]
The absolute integrals converge for $\operatorname{Re}s>-1$, and this transform is holomorphic and nonzero on $-1/2\leq\operatorname{Re}s<0$. These are the required hypotheses. The direct proof above is therefore a concrete realization of that established criterion, supplemented here by the explicit residue analysis and the quantitative conclusions of Sections~\ref{sec:zerosum}--\ref{sec:criterion}.

\section{Concluding remarks}\label{sec:remarks}

The double-pole calculation illustrates a general local mechanism. If a kernel's Mellin transform has a simple pole at a trivial zero of $\zeta$, its quotient by $\zeta$ has a double pole and hence a logarithmic residue. A regular nonzero Mellin transform gives a simple pole instead; a simple zero cancels the zeta zero.

For the family $e^{-n^\alpha x}$, $\alpha>0$, the Mellin quotient is $\Gamma(s)/\zeta(\alpha s)$. Collisions occur when $\alpha k=2m$ with $k,m\geq1$, hence at some negative integers precisely for rational $\alpha$. Uniform formulas near these collisions require additional analysis. Character weights similarly lead to $\Gamma(s)/L(s,\chi)$; primitive and imprimitive characters must be distinguished because of the extra Euler factors in the latter case; compare \cite{GM2023}.

For the transform studied here, the remaining issue in the converse criterion is the uniform estimate $(H_{\mathrm F})$, not pointwise convergence of a smoothed zero sum. The relation to the Mertens formula through partial summation is treated separately; no unsmoothed endpoint estimate is inferred from the present results.

\section*{Funding}
The author declares that no funds, grants, or other support were received during the preparation of this manuscript.

\section*{Acknowledgements}
The author acknowledges mathematical and editorial assistance from Claude (Anthropic). Claude assisted with the uniform Gamma estimate in Lemma~\ref{lem:gammalower}, the M\"obius--Bessel expansions in Theorem~\ref{thm:bessel}, the contour estimates, and the organization of references and cross-references. The author is responsible for the mathematical arguments, source attribution, and final text.

\appendix

\section{Mellin inversion}\label{app:mellininv}

This and the following appendices record, in complete step-by-step detail, the auxiliary results supporting the main text. Throughout, $m,N\in\{1,2,3,\dots\}$, $x>0$, $\gamma_{e}$ is Euler's constant, and $\psi=\Gamma'/\Gamma$. We begin with the inversion lemma invoked in Section~\ref{sec:mellin}.

\begin{lem}[Mellin inversion]\label{lem:mellininv}
Let $f\colon(0,\infty)\to\Real$ be continuous and of bounded variation on every compact subinterval, and suppose that $\intzi\abs{f(y)}\,y^{\sigma-1}\,dy<\infty$ for every $\sigma$ in an open interval $I$, and put
\[F(s)\;=\;\intzi f(y)\,y^{s-1}\,dy\qquad(\operatorname{Re}s\in I).\]
Then, for every $\sigma\in I$ and every $x>0$,
\[f(x)\;=\;\frac{1}{2\pi i}\,\lim_{T\to\infty}\int_{\sigma-iT}^{\sigma+iT}F(s)\,x^{-s}\,ds .\]
\end{lem}

\begin{proof}
Substitute $y=e^{-u}$ in $F$ and set $g(u)=f(e^{-u})\,e^{-\sigma u}$. Then $g$ is continuous on $\Real$, of bounded variation on every compact interval, integrable by the hypothesis at $\sigma$, and
\[F(\sigma+it)\;=\;\intpminf f(e^{-u})\,e^{-(\sigma+it)u}\,du\;=\;\intpminf g(u)\,e^{-itu}\,du\;=\;\widehat{g}(t) .\]
The classical Fourier inversion theorem for continuous, integrable functions of locally bounded variation (see \cite{Donoghue} and \cite{TitchmarshFI}, where the theorem appears with Jordan's test) gives, at every $u\in\Real$,
\[g(u)\;=\;\frac{1}{2\pi}\,\lim_{T\to\infty}\int_{-T}^{T}\widehat{g}(t)\,e^{itu}\,dt ;\]
writing $x=e^{-u}$, so that $e^{itu}=x^{-it}$ and $g(u)=f(x)\,x^{\sigma}$, and dividing by $x^{\sigma}$, this is exactly the display.
\end{proof}

\section{Sine estimates on the contours}\label{app:sines}

We derive the four estimates \eqref{eq:sinexact}, \eqref{eq:sinexact2}, \eqref{eq:sinhoriz}, \eqref{eq:sinhoriz2}. On the vertical line $s=\tfrac12-2N+it$,
\begin{equation}\label{eq:app-vert-angle}
\pi\Big(\tfrac12-2N+it\Big)=\frac{\pi}{2}-2\pi N+i\pi t ,
\end{equation}
and by $\sin(\tfrac{\pi}{2}+z)=\cos z$, the $2\pi$-periodicity of sine, and $\cos(i\pi t)=\cosh(\pi t)$,
\begin{equation}\label{eq:app-vert-sin}
\sin\Big(\pi\big(\tfrac12-2N+it\big)\Big)=\cos(i\pi t)=\cosh(\pi t),
\end{equation}
so that, by $\cosh u=\tfrac12(e^{u}+e^{-u})\geq\tfrac12 e^{\abs{u}}$,
\begin{equation}\label{eq:app-vert-bound}
\Big|\sin\Big(\pi\big(\tfrac12-2N+it\big)\Big)\Big|=\cosh(\pi t)\geq\tfrac12\,e^{\pi\abs{t}} ,
\end{equation}
which is \eqref{eq:sinexact}. For the half-angle, $\tfrac{\pi}{2}\big(\tfrac12-2N+it\big)=\tfrac{\pi}{4}-\pi N+i\tfrac{\pi t}{2}$, and the identity $\abs{\sin(a+ib)}^{2}=\sin^{2}a+\sinh^{2}b$ (which is \eqref{eq:sinsigit2}) gives, with $\sin^{2}\big(\tfrac{\pi}{4}-\pi N\big)=\tfrac12$ and $1+2\sinh^{2}u=\cosh(2u)$,
\begin{equation}\label{eq:app-half-exact}
\Big|\sin\Big(\tfrac{\pi}{2}\big(\tfrac12-2N+it\big)\Big)\Big|^{2}
=\tfrac12+\sinh^{2}\Big(\tfrac{\pi t}{2}\Big)=\tfrac12\cosh(\pi t)\geq\tfrac12 ,
\end{equation}
which is \eqref{eq:sinexact2}. On the horizontal lines $s=\sigma\pm iT$ with $T\geq1$, the same identity with $a=\pi\sigma$, $b=\pm\pi T$ gives
\begin{equation}\label{eq:app-horiz}
\abs{\sin(\pi(\sigma\pm iT))}^{2}=\sin^{2}(\pi\sigma)+\sinh^{2}(\pi T)\geq\sinh^{2}(\pi T),
\end{equation}
and, since $T\geq1$,
\[\sinh(\pi T)\;=\;\tfrac12\big(e^{\pi T}-e^{-\pi T}\big)
\;\geq\;\tfrac12\,e^{\pi T}\big(1-e^{-2\pi}\big)\;\geq\;\tfrac14\,e^{\pi T} ,\]
which is \eqref{eq:sinhoriz}. Identically,
\[\abs{\sin\Big(\tfrac{\pi}{2}(\sigma\pm iT)\Big)}\;\geq\;\sinh\Big(\tfrac{\pi T}{2}\Big)
\;\geq\;\tfrac12\,e^{\pi T/2}\big(1-e^{-\pi}\big)\;\geq\;\tfrac14\,e^{\pi T/2} ,\]
which is \eqref{eq:sinhoriz2}.

\section{Zeta values used in the main text}\label{app:zetavalues}

We record $\zeta(0)=-\tfrac12$, the trivial zeros $\zeta(-2m)=0$, and Euler's evaluation \eqref{eq:euler2m}, equivalently
\[\zeta(2m)\;=\;\frac{\abs{B_{2m}}\,(2\pi)^{2m}}{2\,(2m)!}\,;\]
for example,
\[\zeta(2)=\frac{\pi^{2}}{6},\qquad\zeta(4)=\frac{\pi^{4}}{90},\qquad\zeta(6)=\frac{\pi^{6}}{945} .\] All are classical (\cite[Chapter II]{Titchmarsh}); $\zeta(0)=-\tfrac12$ also follows from the functional equation \eqref{eq:funceq} by letting $s\to0$ and using $\zeta(1-s)\sim-1/s$ together with $\sin(\tfrac{\pi s}{2})\sim\tfrac{\pi s}{2}$:
\begin{equation}\label{eq:app-zeta0}
\zeta(0)=\lim_{s\to0}2^{s}\pi^{s-1}\cdot\frac{\pi s}{2}\cdot\Gamma(1-s)\cdot\Big(-\frac1s\Big)=-\frac12 .
\end{equation}

\section{\texorpdfstring{Laurent expansion of the Gamma function at $s=-n$}{Laurent expansion of the Gamma function at s=-n}}\label{app:gammalaurent}

We prove \eqref{eq:gammalaurent}: for every integer $n\geq0$, as $s\to-n$, with $w=s+n$,
\begin{equation}\label{eq:app-gammalaurent}
\Gamma(s)=\frac{(-1)^{n}}{n!}\left(\frac{1}{w}+\psi(n+1)+O(w)\right).
\end{equation}
Iterating the recurrence $\Gamma(z+1)=z\Gamma(z)$ exactly $n+1$ times,
\begin{equation}\label{eq:app-gammarec}
\Gamma(s)=\frac{\Gamma(s+n+1)}{s(s+1)\cdots(s+n)}=\frac{\Gamma(1+w)}{(w-n)(w-n+1)\cdots(w-1)\,w} .
\end{equation}
The numerator expands as $\Gamma(1+w)=1-\gamma_{e}w+O(w^{2})$, by $\Gamma'(1)=-\gamma_{e}$. The product of the $n$ nonvanishing factors in the denominator is
\begin{equation}\label{eq:app-denomprod}
\prod_{j=1}^{n}(w-j)=(-1)^{n}\,n!\,\prod_{j=1}^{n}\Big(1-\frac{w}{j}\Big)
=(-1)^{n}\,n!\,\big(1-H_{n}\,w+O(w^{2})\big),
\end{equation}
where
\[H_{n}\;=\;\sum_{j=1}^{n}\frac{1}{j}\]
is the $n$-th harmonic number, with the empty-sum convention $H_{0}=0$. Therefore
\begin{align}
\Gamma(s)&=\frac{1-\gamma_{e}w+O(w^{2})}{(-1)^{n}n!\,w\,\big(1-H_{n}w+O(w^{2})\big)}\label{eq:app-gammaassemble}\\
&=\frac{(-1)^{n}}{n!}\cdot\frac{1}{w}\big(1+(H_{n}-\gamma_{e})w+O(w^{2})\big),\notag
\end{align}
and since $\psi(n+1)=H_{n}-\gamma_{e}$ (\cite[Section 5.4]{DLMF}), this is \eqref{eq:app-gammalaurent}. In the application to the trivial zeros we take $n=2m$, where $(-1)^{2m}=1$:
\begin{equation}\label{eq:app-gamma2m}
\Gamma(s)=\frac{1}{(2m)!}\left(\frac{1}{w}+\psi(2m+1)+O(w)\right)\qquad(w=s+2m\to0).
\end{equation}

\section{\texorpdfstring{Expansion of $1/\zeta$ and of $x^{-s}$ at a trivial zero}{Expansion of 1/zeta and of x to the power (-s) at a trivial zero}}\label{app:invzeta}

Since $\zeta(-2m)=0$ and the trivial zeros are simple, $\zeta'(-2m)\neq0$, and Taylor expansion at $-2m$ gives, with $w=s+2m$,
\begin{equation}\label{eq:app-zetataylor}
\zeta(s)=\zeta'(-2m)\,w\left(1+\frac{\zeta''(-2m)}{2\,\zeta'(-2m)}\,w+O(w^{2})\right),
\end{equation}
whence, inverting the unit factor by the geometric series,
\begin{equation}\label{eq:app-invzeta}
\frac{1}{\zeta(s)}=\frac{1}{\zeta'(-2m)\,w}\left(1-\frac{\zeta''(-2m)}{2\,\zeta'(-2m)}\,w+O(w^{2})\right).
\end{equation}
Also $x^{-s}=x^{2m-w}=x^{2m}e^{-w\log x}$, so
\begin{equation}\label{eq:app-xexp}
x^{-s}=x^{2m}\big(1-w\log x+O(w^{2})\big).
\end{equation}

\section{\texorpdfstring{Direct Laurent computation of the residue at $s=-2m$}{Direct Laurent computation of the residue at s=-2m}}\label{app:laurentres}

Multiplying \eqref{eq:app-gamma2m}, \eqref{eq:app-invzeta} and \eqref{eq:app-xexp},
\begin{multline}\label{eq:app-product}
\frac{\Gamma(s)\,x^{-s}}{\zeta(s)}
=\frac{x^{2m}}{(2m)!\ \zeta'(-2m)}\cdot\frac{1}{w^{2}}
\Big(1+\psi(2m+1)\,w+O(w^{2})\Big)\\
\times\Big(1-\frac{\zeta''(-2m)}{2\,\zeta'(-2m)}\,w+O(w^{2})\Big)
\Big(1-w\log x+O(w^{2})\Big).
\end{multline}
The coefficient of $w$ in the product of the three unit factors is $\psi(2m+1)-\frac{\zeta''(-2m)}{2\zeta'(-2m)}-\log x$, so the coefficient of $w^{-1}$ in \eqref{eq:app-product}, that is the residue, equals
\begin{equation}\label{eq:app-rawres}
\Res_{s=-2m}\frac{\Gamma(s)\,x^{-s}}{\zeta(s)}
=\frac{x^{2m}}{(2m)!\ \zeta'(-2m)}
\left(\psi(2m+1)-\frac{\zeta''(-2m)}{2\,\zeta'(-2m)}-\log x\right),
\end{equation}
which is the residue whose agreement with \eqref{eq:evenres} is verified in Appendix \ref{app:sanity}.

\section{\texorpdfstring{Functional-equation evaluation of $\zeta'(-2m)$ and $\zeta''(-2m)$}{Functional-equation evaluation of zeta'(-2m) and zeta''(-2m)}}\label{app:funceqderiv}

Write the functional equation \eqref{eq:funceq} as the factorization
\begin{equation}\label{eq:app-hdef}
\zeta(s)=h(s)\,\sin\big(\tfrac{\pi s}{2}\big),\qquad
h(s)=2^{s}\pi^{s-1}\,\Gamma(1-s)\,\zeta(1-s),
\end{equation}
where $h$ is analytic and nonzero in a neighbourhood of $s=-2m$: indeed $\Gamma(1-s)$ and $\zeta(1-s)$ are analytic and nonzero at $s=-2m$, their arguments being $2m+1>1$. With $w=s+2m$,
\begin{equation}\label{eq:app-sinexp}
\sin\big(\tfrac{\pi s}{2}\big)=\sin\big(\tfrac{\pi w}{2}-\pi m\big)
=(-1)^{m}\sin\big(\tfrac{\pi w}{2}\big)
=(-1)^{m}\Big(\tfrac{\pi}{2}w+O(w^{3})\Big),
\end{equation}
and $h(s)=h(-2m)+h'(-2m)\,w+O(w^{2})$, so
\begin{equation}\label{eq:app-zetaexp2}
\zeta(s)=(-1)^{m}\tfrac{\pi}{2}\,h(-2m)\,w+(-1)^{m}\tfrac{\pi}{2}\,h'(-2m)\,w^{2}+O(w^{3}).
\end{equation}
Comparing with \eqref{eq:app-zetataylor} coefficient by coefficient,
\begin{align}
\zeta'(-2m)&=(-1)^{m}\tfrac{\pi}{2}\,h(-2m),\qquad
\frac{\zeta''(-2m)}{2}=(-1)^{m}\tfrac{\pi}{2}\,h'(-2m),\label{eq:app-zpid}\\
\frac{\zeta''(-2m)}{2\,\zeta'(-2m)}&=\frac{h'(-2m)}{h(-2m)} .\notag
\end{align}
Since
\[h(-2m)\;=\;2^{-2m}\,\pi^{-2m-1}\,(2m)!\ \zeta(2m+1) ,\]
the first identity gives the closed form
\begin{equation}\label{eq:app-zpval}
\zeta'(-2m)=\frac{(-1)^{m}\,(2m)!\ \zeta(2m+1)}{2\,(2\pi)^{2m}} .
\end{equation}
For the logarithmic derivative, $\log h(s)=s\log2+(s-1)\log\pi+\log\Gamma(1-s)+\log\zeta(1-s)$ on a simply connected neighbourhood of $-2m$ avoiding zeros of $h$, so
\begin{equation}\label{eq:app-hlog}
\frac{h'(s)}{h(s)}=\log(2\pi)-\psi(1-s)-\frac{\zeta'(1-s)}{\zeta(1-s)} ,
\end{equation}
and evaluating at $s=-2m$ and inserting into \eqref{eq:app-zpid} gives
\begin{equation}\label{eq:app-zppval}
\frac{\zeta''(-2m)}{2\,\zeta'(-2m)}=\log(2\pi)-\psi(2m+1)-\frac{\zeta'(2m+1)}{\zeta(2m+1)} .
\end{equation}

\section{Agreement of the two residue computations}\label{app:sanity}

Substituting \eqref{eq:app-zppval} into \eqref{eq:app-rawres}, the bracket becomes
\begin{multline}\label{eq:app-bracket}
\psi(2m+1)-\Big(\log(2\pi)-\psi(2m+1)-\frac{\zeta'(2m+1)}{\zeta(2m+1)}\Big)-\log x\\
=2\psi(2m+1)+\frac{\zeta'(2m+1)}{\zeta(2m+1)}-\log(2\pi x),
\end{multline}
while \eqref{eq:app-zpval} gives
\begin{equation}\label{eq:app-prefactor}
\frac{1}{(2m)!\ \zeta'(-2m)}=\frac{2\,(-1)^{m}\,(2\pi)^{2m}}{[(2m)!]^{2}\ \zeta(2m+1)} .
\end{equation}
Multiplying \eqref{eq:app-bracket} by \eqref{eq:app-prefactor} and by $x^{2m}$, and using $\psi(2m+1)=\Gamma'(2m+1)/(2m)!$, we obtain exactly \eqref{eq:evenres}. The two computations of the residue at the double pole, the sine-splitting computation of Section \ref{sec:even} and the direct Laurent computation above, therefore agree identically.

\section{The Gamma lower bound from the Weierstrass product}\label{app:gammaproduct}

We supply every step of Lemma \ref{lem:gammalower}. The Weierstrass product (\cite[Chapter XII]{WW}) is
\begin{equation}\label{eq:app-weier}
\frac{1}{\Gamma(z)}=z\,e^{\gamma_{e}z}\prod_{n\geq1}\Big(1+\frac{z}{n}\Big)e^{-z/n},
\end{equation}
convergent locally uniformly. At $z=\sigma+it$, $\sigma>0$, taking squared moduli and using $\abs{e^{\gamma_{e}z}}^{2}=e^{2\gamma_{e}\sigma}$, $\abs{e^{-z/n}}^{2}=e^{-2\sigma/n}$, $\abs{z}^{2}=\sigma^{2}+t^{2}$ and $\abs{n+z}^{2}=(n+\sigma)^{2}+t^{2}$,
\begin{equation}\label{eq:app-modsq}
\frac{1}{\abs{\Gamma(\sigma+it)}^{2}}
=(\sigma^{2}+t^{2})\,e^{2\gamma_{e}\sigma}\prod_{n\geq1}\frac{(n+\sigma)^{2}+t^{2}}{n^{2}}\,e^{-2\sigma/n} .
\end{equation}
Dividing the value of \eqref{eq:app-modsq} at $t=0$ by its value at $t$, every $t$-independent factor cancels, and
\begin{equation}\label{eq:app-ratio}
\frac{\abs{\Gamma(\sigma+it)}^{2}}{\Gamma(\sigma)^{2}}
=\frac{\sigma^{2}}{\sigma^{2}+t^{2}}\prod_{n\geq1}\frac{(n+\sigma)^{2}}{(n+\sigma)^{2}+t^{2}}
=\prod_{n\geq0}\Big(1+\frac{t^{2}}{(n+\sigma)^{2}}\Big)^{-1},
\end{equation}
which is \eqref{eq:gammaproduct}. Taking logarithms and comparing the sum with the integral of the positive decreasing function $u\mapsto\log(1+t^{2}/u^{2})$ over $u\in(\sigma,\infty)$, term $n$ being bounded by the integral over~$(\sigma+n-1,\sigma+n)$ for $n\geq1$,
\begin{equation}\label{eq:app-sumint}
\sum_{n\geq0}\log\Big(1+\frac{t^{2}}{(n+\sigma)^{2}}\Big)
\leq\log\Big(1+\frac{t^{2}}{\sigma^{2}}\Big)+\int_{0}^{\infty}\log\Big(1+\frac{t^{2}}{u^{2}}\Big)du .
\end{equation}
For $t=0$ the integrand vanishes identically and the integral equals $\pi\abs{t}=0$; assume therefore $t\neq0$. In the integral substitute $u=\abs{t}v$ and integrate by parts:
\begin{equation}\label{eq:app-ibp}
\int_{0}^{\infty}\log\big(1+v^{-2}\big)\,dv
=\Big[v\log\big(1+v^{-2}\big)\Big]_{0}^{\infty}+\int_{0}^{\infty}\frac{2\,dv}{1+v^{2}}=0+\pi ,
\end{equation}
the boundary term vanishing at both ends because
\begin{align*}
v\,\log\big(1+v^{-2}\big)&\;\sim\;2\,v\,\log\frac{1}{v}\;\longrightarrow\;0\qquad(v\to0^{+}),\\
v\,\log\big(1+v^{-2}\big)&\;\sim\;\frac{1}{v}\;\longrightarrow\;0\qquad(v\to\infty).
\end{align*} Hence the integral in \eqref{eq:app-sumint} equals $\pi\abs{t}$, and exponentiating,
\begin{equation}\label{eq:app-final}
\frac{\abs{\Gamma(\sigma+it)}^{2}}{\Gamma(\sigma)^{2}}
\geq\Big(1+\frac{t^{2}}{\sigma^{2}}\Big)^{-1}e^{-\pi\abs{t}} ,
\end{equation}
which is the square of \eqref{eq:gammalower}.

\end{document}